\documentclass[twoside,12pt]{article}
\usepackage{graphicx}
\usepackage{subfigure}
\usepackage{url}
\usepackage{color}
\usepackage{amsmath}
\usepackage{amsthm}

\usepackage{amsfonts}

\newtheorem{theorem}{Theorem}

\newtheorem{lemma}[theorem]{Lemma}
\newtheorem{remark}{Remark}

\title{On 1:3 resonance under reversible perturbations of conservative cubic H\'enon maps}

\author{
M.\,S.\,Gonchenko$^1$, A.\,O.\,Kazakov$^2$, E.\,A.\,Samylina$^{2,3}$, A.\,I.\,Shykhmamedov$^2$		\\%[12pt]
{\small		$^1$ Universitat Polit\`ecnica de Catalunya, Barcelona, Spain}\\
{\small		$^2$ National Research University Higher School of Economics, Nizhny Novgorod, Russia}\\
{\small		$^3$ Lobachevsky State University of Nizhny Novgorod, Nizhny Novgorod, Russia
}\\
{\small { \texttt{marina.gonchenko@upc.edu}, 	 \texttt{kazakovdz@yandex.ru}, 			 \texttt{samylina\_evgeniya@mail.ru},				 
\texttt{aykhansh@gmail.com}	 } }
}

\date{\today}

\begin{document}
	
	\maketitle
\begin{abstract}

%Variant 1:\\
We consider reversible non-conservative perturbations of  the conservative cubic H\'enon maps $H_3^{\pm}: \bar x = y, \bar  y = -x + M_1 + M_2 y \pm y^3$ and study their influence on the 1:3 resonance, i.e. bifurcations of fixed points with eigenvalues~$e^{\pm i 2\pi/3}$. It follows from~\cite{DulinMeiss2000}, this resonance is degenerate for~$M_1=0, M_2=-1$ when the corresponding fixed point is elliptic. We show that bifurcations of this point %(when~$M_1$ changes and~$M_2$ is fixed) 
under reversible perturbations give rise to four 3-periodic orbits, two of them are symmetric and conservative (saddles in the case of map~$H_3^+$ and elliptic orbits in the case of map~$H_3^-$), the other two orbits are nonsymmetric and they compose symmetric couples of dissipative orbits (attracting and repelling orbits in the case of map~$H_3^+$ and saddles with the Jacobians less than~1 and greater than~1 in the case of map~$H_3^-$). We show that these local symmetry-breaking bifurcations can lead to mixed dynamics due to accompanying global reversible bifurcations of symmetric non-transversal homo- and heteroclinic cycles. We also generalize the results of~\cite{DulinMeiss2000} to the case of the~$p:q$ resonances with odd~$q$ and show that all of them are also degenerate for the %unperturbed 
maps~$H_3^{\pm}$ with~$M_1=0$.

\end{abstract}

\section{Introduction}
	
In the present paper we study how reversible non-conservative perturbations affect the 1:3 resonance, i.e. bifurcations of fixed points with eigenvalues~$e^{\pm i 2\pi/3}=-1/2\pm i \sqrt{3}/2$, in the conservative cubic H\'enon maps
\begin{equation}
	H^{\pm}_3: \; (x,y)\to (\bar x, \bar y)\;:
	%\begin{cases}
	\bar x = y,\;\;\;
	\bar y =- x +  M_1 + M_2 y \pm y^3,
	%\end{cases}
	\label{eq:HenonMapCubic}
\end{equation}
where~$x$ and~$y$ are coordinates,~$M_1$ and~$M_2$ are real coefficients. This problem is of great interest since it is closely related to the so-called phenomenon of mixed dynamics~\cite{Gon2016, GonTur2017, GonchenkoGK20}, the third recently discovered type of chaos (in addition to the well-known conservative and dissipative chaos) which is characterized by the principal inseparability of conservative and dissipative elements of dynamics. Indeed, the cubic H\'enon maps~(\ref{eq:HenonMapCubic}) appear as truncated first return maps near cubic homoclinic tangencies in area-preserving diffeomorphisms~\cite{GonGonOvs2017}. Note that the different signs~$\pm 1$ before the cubic term~$y^3$ correspond to two different types of cubic homoclinic tangencies, see~\cite{Gonchenko85,GonSimoViero13, GonGonOvs2017} for more details. In the case of reversible maps, the appearance of such symmetric homoclinic tangencies is a codimension~1 bifurcation phenomenon~\cite{GonchenkoGSin20}. On the other hand, the existence of homoclinic tangencies implies the existence of Newhouse domains where such tangencies, including symmetric cubic homoclinic tangencies in the reversible case, are dense~\cite{GonchenkoGSin20}. In turn, the emergence of symmetric pairs of nonsymmetric orbits provides a criterion of the appearance of mixed dynamics. Namely, in these Newhouse domains, maps with infinitely many periodic sinks (also called stable or attracting orbits), sources (completely unstable or repelling orbits), area-expanding saddles (with the Jacobian greater than 1) and area-contracting saddles (with the Jacobian less than 1) as well as symmetric elliptic periodic orbits are dense.
	
It is widely known that the strong 1:1, 1:2, 1:3 and 1:4 resonances, i.e. bifurcations of fixed points (periodic orbits) with eigenvalues~$e^{\pm 2\pi i/q}$,~$q=1,2,3,4$, respectively, are very important for dynamics. In the conservative setting, the nondegenerate 1:1 resonance is related to a parabolic (elliptic-hyperbolic) bifurcation of fixed (periodic) points that implies the appearance of a pair of saddle and elliptic orbits. In turn, the nondegenerate 1:2 resonance is connected with a conservative period-doubling bifurcation. The 1:3 and 1:4 resonances are most difficult, their theory was outlined in~\cite{Arn-Geom}, where, in particular, a case of the degenerate 1:4 resonance (the so-called ``Arnold degeneracy'') was considered, see also~\cite{ArnKozNeisht07}. A new, second, type of degenerate conservative 1:4 resonance has been found recently in~\cite{GonGonOvs2017, GGOV18}.
As well-known, the complex local normal form of an area-preserving map near a fixed point with eigenvalues~$\pm i$ (1:4 resonance) is written as~$\bar z = i (z + A |z|^2 z + B (z^*)^3) + O(|z|^5)$, where~$A$ and~$B$ are real coefficients. The Arnold degeneracy corresponds to the case~$|A| = |B|$, the other degeneracy from~\cite{GonGonOvs2017} is related to the case~$B=0$. The latter degeneracy is very interesting since it is accompanied by symmetry-breaking (pitchfork) bifurcations of 4-periodic orbits. Note that a class of degenerate~$p:q$ resonances accompanied by symmetry-breaking bifurcations of~$q$-periodic points was described in~\cite{GLRT14}. The 1:4 resonance fits well into this class~\cite{GonGonOvs2017}, while the 1:3 resonance has certain peculiarities. 
In particular, when the map possesses the central symmetry, four 3-periodic orbits appear near the 1:3 resonant fixed point: two of them are elliptic while the other two are saddle. In the present paper we also show that all~$p:q$ resonances with odd~$q\geq 3$ are degenerate for the cubic H\'enon maps with the central symmetry (when~$M_1=0$).
	
The strong resonances often appear in area-preserving maps. For example, in the conservative quadratic H\'enon map~$\bar x = y, \bar y = M - x - y^2$ the structure of the 1:4 resonance was studied in~\cite{Bir87}, where it was shown that this resonance is degenerate (the Arnold case). In~\cite{Bir87} it was also shown that in the quadratic H\'enon map, the 1:3 resonance is nondegenerate, and it mainly consists of the rearrangement of symmetric 3-periodic saddle orbits. Bifurcations in two-parameter families of area-preserving H\'enon-like maps were also studied in~\cite{SV09}.  In particular, in~\cite{SV09} it was demonstrated that the emergence of fixed point with eigenvalues~$e^{\pm i 2\pi/3}$ in the conservative quadratic H\'enon map implies both local instability of the fixed point and global instability of the map, i.e the fixed point becomes saddle with 6 separatrices (local effect) and  almost all orbits close to the fixed point go to infinity (global effect). However, this is not the case when the 1:3 resonance is degenerate. Here, in general, the fixed point is surrounded by a garland (a chain of stability islands) which consists of elliptic and saddle 3-periodic orbits and does not allow orbits to pass far away from the fixed point. This local stability also implies global stability when the 1:3 resonance is near-degenerate. Note that such a situation takes place in the conservative cubic H\'enon maps~(\ref{eq:HenonMapCubic}), see Figures~\ref{Fig:HenonCubicBD_H3+} and~\ref{Fig:HenonCubicBD_H3-}. Here a degenerate 1:3 resonance appears at~$M_1=0$ and~$M_2=-1$, otherwise it is nondegenerate if~$M_1\neq 0$. We also note that all~$p:q$ resonances with odd~$q$ have the same nature: they are degenerate for~$M_1=0$ and the corresponding value of~$M_2$, see Figure~\ref{fig:FigPQRes} for the 1:5 and 1:7 resonances.
	
Our main goal is to study the near-degenerate 1:3 resonance in~(\ref{eq:HenonMapCubic}) and analyze how it is impacted by reversible non-conservative perturbations. It follows from~\cite{DulinMeiss2000} that for~$M_1=0$ and~$M_2=-1$, the~$1:3$ resonance is degenerate for the conservative cubic H\'enon maps and 3-periodic orbits undergo pitchfork bifurcations. We note that both maps~$H^{+}_3$ and~$H^{-}_3$ are conservative and reversible with respect to the involution~$h: (x,y) \to (y, x)$, i.e. by definition the maps~$H^{\pm}_3$ and the inverse maps~$(H^{\pm}_3)^{-1}$ are conjugate by means of the involution~$h$ (the relation~$(H^{\pm}_3)^{-1} = h\circ H^{\pm}_3 \circ h$ holds). In the general case, reversible maps can also have dissipative orbits that always exist in pairs: stable and completely unstable periodic orbits, two saddle periodic orbits with the Jacobian greater than~1 and less than~1, etc. Such orbits are symmetric to each other with respect to the involution. We call them a \emph{symmetric couple of orbits}. When a map is reversible and conservative, symmetric couples of orbits are conservative, however, under general reversible perturbations, these pairs can become dissipative.
		
The genericity of perturbations means that they should destroy the conservativity. Following the paper~\cite{GonchenkoGSaf20}, we construct (analytically) such reversibility preserving perturbations. We present two direct methods to obtain such perturbations for the conservative cubic H\'enon maps~(\ref{eq:HenonMapCubic}). The first method gives reversible non-conservative perturbations in the so-called cross-form~(\ref{eq:RevMap}), and, in this way, we obtain perturbations of the second iteration of the inverse cubic maps~$(H^{\pm}_3)^{-2}$ in the form~(\ref{eq:HSqPerturb}). The second method provides a perturbation of the cubic maps themselves, see the map~(\ref{eq:rev2cub}). We illustrate the bifurcation diagrams for the near-degenerate 1:3 resonance in the perturbed maps and focus on pitchfork bifurcations of 3-periodic orbits which lead to the dissipative dynamics. Namely, we demonstrate that for perturbations of~$H^+_3$ a supercritical pitchfork bifurcation takes place which leads to the emergence of nonsymmetric attracting and repelling 3-periodic orbits, while for perturbed map~$H^-_3$ under a subcritical pitchfork bifurcation there appears a pair of nonsymmetric saddles, one is with the Jacobian greater than~1 and the other is with the Jacobian less than~1.

The fact that mixed dynamics can often appear in applications was shown in~\cite{GonGonKaz2013, GonGonKazTur2017}. Recently, the number of the related results has sharply increased, see e.g.~\cite{Kuz17, Kuz18, EmelianovaN19, Kaz19, ArSch20, EmelianovaN20, BizMam20a, GonchenkoGK20, BizMam20b, EmelianovaN21}. In all these studies, two sides of the phenomenon of mixed dynamics come to light: first, to find it numerically or experimentally, i.e. to provide evidence that attractors and repellers intersect, and, second, to prove it mathematically. As far as we know, only in a couple of papers, see e.g.~\cite{GonGonKazTur2017,Kaz20}, both sides were analyzed. The second side (to prove) is much delicate than the first one, it requires involving various theoretical aspects of mixed dynamics such as criteria for the existence of absolute Newhouse regions~\cite{GonShilTur97, LambStenkin2004, Turaev10, GonchenkoGSin20} and the structure of bifurcation scenarios leading to the appearance of mixed dynamics, see e.g.~\cite{GonGonKazTur2017, Kaz19, Kaz20}. In the present paper we apply both these approaches when studying local and global bifurcations.
	
The paper is organized as follows. In Section~\ref{sec:rsMDsbb} we present the main elements of the theory of reversible systems, mixed dynamics and the associated symmetry-breaking bifurcations that we will use throughout the paper. In Section~\ref{sec:constrRNP} we discuss the two methods to construct reversible non-conservative perturbations of the conservative H\'enon-like maps and prove that the perturbed maps are indeed reversible. In this way, we obtain perturbations of the conservative cubic H\'enon maps~(\ref{eq:HenonMapCubic}). In Section~\ref{sec:1p3resUnpert} we review the structure of bifurcations of the conservative 1:3 resonance and mention the associated degeneracies and pitchfork bifurcations of 3-periodic orbits for the unperturbed cubic maps~$H^{\pm}_3$ in the form~(\ref{eq:HSqPerturb}). Also in this section we show that all~$p:q$ resonances with odd~$q$ are triple degenerate in maps~$H_3^{\pm}$ with~$M_1=0$.  In Sections~\ref{sec:H_3plus} and~\ref{sec:H_3minus} we analyze reversible and non-conservative perturbations of~$H^+_3$ and~$H^-_3$, respectively, as well as  display bifurcation diagrams for near-degenerate 1:3 resonance and demonstrate the appearance of dissipative 3-periodic orbits under pitchfork bifurcations in both cases. In Section~\ref{sec:md} we provide numerical evidence of mixed dynamics in a reversible non-conservative perturbation of~$H^-_3$ and discuss possible emergence of mixed dynamics in the map~$H^+_3$.

\section{Reversible systems, mixed dynamics and symmetry-breaking bifurcations}\label{sec:rsMDsbb}

In~\cite{GonTur2017}, it was established that multidimensional systems with the compact phase space can have three different and independent forms of dynamics. Two of them have been well known for a long time: these are {\em conservative} and {\em dissipative} dynamics, the third form, the so-called {\em mixed dynamics}, is rather new.

The most famous example of conservative dynamics is provided by systems which preserve phase volume (e. g. Hamiltonian systems, area- and volume-preserving maps). From the point of view of topological dynamics, the conservative dynamics is characterized by the fact that the phase space is chain-transitive~\cite{AnBr85,GonTur2017}, i.e. any two points can be connected by~$\varepsilon$-orbits for any~$\varepsilon > 0$. The dissipative dynamics has a completely different nature: the phase space is not chain-transitive and, besides, one can construct a set  of nonintersecting absorbing and repelling domains containing, respectively, all attractors and repellers of the system.

As for the mixed dynamics,  first of all, it is characterized by the principal inseparability of attractors and repellers~\cite{GonShilTur97}, which implies the existence of infinitely many dissipative attractors and repellers and the impossibility of constructing a set of disjoint absorbing and repelling domains, see more details in~\cite{GonTur2017,GonchenkoGK20}.

As well-known~\cite{Conley78}, any homeomorphism of a compact phase space has attractors and repellers. An attractor is considered in the Ruelle sense~\cite{Ruelle81} as a closed invariant and chain-transitive  set which is stable under permanently acting perturbations\footnote{This type of stability is also called total stability or Lyapunov stability for~$\varepsilon$-orbits and has been known for a very long time, see e.g.~\cite{Mal44}. In fact, this means the stability under arbitrary small bounded noise.}, and a repeller is an attractor of the inverse map. In~\cite{GonTur2017} it was shown that such attractors and repellers can have a non-empty intersection, in contrast to attractors and repellers that are attracting and repelling sets, i.e. asymptotically stable invariant sets under forward and backward iterations, respectively. Denote the full sets of such attractors and repellers of the phase space as~${\cal A}$ and~${\cal R}$, respectively. Such closed invariant sets are called \emph{the full attractor} and \emph{full repeller} of a map in~\cite{GonTur2017}. Then the above classification of dynamics can be built according to the following 
principle~\cite{GonTur2017}: if~${\cal A}={\cal R}$, then the system demonstrates (topologically) conservative dynamics (in this case,~${\cal A}$ and~${\cal R}$ coincide with the whole phase space); if~${\cal A} \cap {\cal R} = \emptyset$,  then the dynamics is dissipative; and ~${\cal A} \cap {\cal R} \neq \emptyset$ and~${\cal A} \neq {\cal R}$ in the case of mixed dynamics. These logical relations constitute the complete system, and, therefore,  there is no other type of dynamics demonstrated by homeomorphisms of compact phase spaces.

However, for concrete systems with chaotic dynamics, it is practically never known precisely what~${\cal A}$ and~${\cal R}$ are. Therefore, certain, efficiently verifiable criteria are needed to determine the type of dynamics. If such criteria are well known for conservative and dissipative dynamics, then, in the case of mixed dynamics, the situation is more complicated: all known criteria are quite nontrivial and are related to the existence of the so-called absolute Newhouse regions~\cite{GonShilTur97, Turaev10, T15}.

In the case of two-dimensional diffeomorphisms, there are such regions where diffeomorphisms with the following properties are generic\footnote{The Newhouse regions are open in~$C^r$-topology with~$r\geq 2$, and a certain property is called generic if it holds for a residual subset (a set of the second Baire category) of diffeomorphisms of such region.}: (i) every such diffeomorphism has infinitely many periodic sinks, sources, and saddles and (ii) the closures of the sets of orbits of different types  have non-empty intersections. In~\cite{GonShilTur97}, it was proven that  the absolute Newhouse regions  exist in any neighborhood of a diffeomorphism with a non-transversal heteroclinic cycle that contains two saddle fixed points~$O_1$ and~$O_2$ and two heteroclinic orbits~$\Gamma_{12}$ and~$\Gamma_{21}$ such that~$W^s(O_1)$ and~$W^u(O_2)$ intersect at the points of~$\Gamma_{12}$ transversally, and~$W^s(O_2)$ and~$W^u(O_1)$ have a quadratic tangency at the points of~$\Gamma_{21}$, see Figure~\ref{Fig:HeteroclonicCycles}a. The principally important condition is that { the Jacobian of the map is greater than 1 at one of the fixed points and less than 1 at the other}. This result shows that the mixed dynamics is a generic property of two-dimensional diffeomorphisms from absolute Newhouse regions: indeed, the sets~${\cal A}$ and~${\cal R}$ are closed invariant sets that should contain all sinks and sources, respectively, and, hence, the set~${\cal A}\cap{\cal R}$ is non-empty due to the above property (ii).

\begin{figure}[t!]
\begin{minipage}[h]{1\linewidth}
\center{\includegraphics[width=0.8\columnwidth]{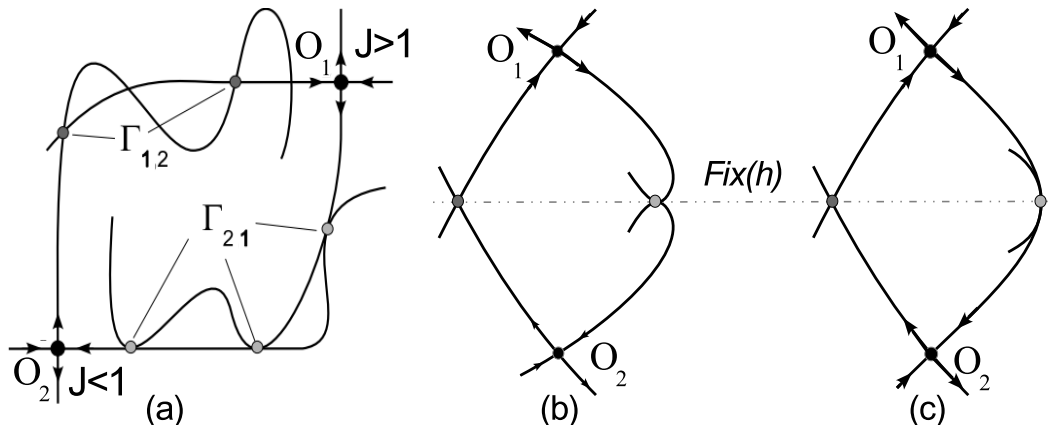}}
\end{minipage}
\caption{{\footnotesize Examples of non-transversal heteroclinic cycles of planar diffeomorphisms: (a) the heteroclinic cycle studied in~\cite{GonShilTur97}; (b) the non-transversal heteroclinic cycle with a quadratic tangency from~\cite{LambStenkin2004}; (c) the non-transversal heteroclinic cycle with a cubic tangency from~\cite{GonchenkoGSin20}.}}
		\label{Fig:HeteroclonicCycles}
\end{figure}

In the present paper we consider two-dimensional reversible diffeomorphisms. Recall that a diffeomorphism~$f$ is reversible if it is conjugate to its inverse map~$f^{-1}$ by means of an involution~$h$, i.e. the following relation is true:~$f^{-1} =  h\circ f\circ  h$, where~$h^2 = Id$.
The property of reversibility of~$f$ implies the strong symmetry of the set of orbits. An orbit that intersects the set~$Fix(R) =\{x: \; R(x) = x\}$ or the set~$Fix(R f)$ is called {\em symmetric}. Any symmetric periodic orbit of a two-dimensional reversible orientable map has eigenvalues~$\lambda$ and~$\lambda^{-1}$.  Moreover, such an orbit with eigenvalues~$e^{\pm i\varphi}$, where $\varphi\neq 0,\pi$, is, essentially, elliptic, since the principal hypotheses of the KAM theory hold~\cite{Sevryuk86}. However, there are principal differences in dynamics. In particular, it was shown in~\cite{GLRT14,GonTur2017} that a generic elliptic point of two-dimensional reversible maps is not completely conservative since it is the limit of periodic sinks and sources.

As for nonsymmetric orbits, they can be, in principle, of arbitrary type. However, for any nonsymmetric orbit, there exists a symmetric to it orbit with ``opposite'' dynamical properties. It means that if a periodic orbit has eigenvalues~$\lambda_1$ and~$\lambda_2$, then the symmetric to it orbit has eigenvalues~$\lambda_1^{-1}$ and~$\lambda_2^{-1}$. As we said before, nonsymmetric orbits compose a {symmetric pair of orbits}.

As in the dissipative case, in the space of reversible systems, Newhouse regions (i.e. such open regions in which reversible systems with both symmetric and nonsymmetric homoclinic tangencies are dense) exist near any system with a symmetric homoclinic tangency. However, there is one nontrivial moment related to the proof of the fact that these regions are absolute Newhouse regions, i.e. that they contain a residual subset of systems having infinitely many coexisting periodic attractors, repellers, saddles and elliptic orbits and the closure of the sets of the orbits of different types has a non-empty intersection.

This problem, proposed in~\cite{GonDelsh2013} as the Reversible Mixed Dynamics conjecture (RMD-conjecture), remains open for the multidimensional case. For two-dimensional reversible maps, it was proved in~\cite{GLRT14} for~$C^r$-perturbations (with~$2\leq r \leq \infty$) that keep the reversibility. Also the RMD-conjecture was proved in~\cite{LambStenkin2004,GonDelsh2013,DGGL18,GonchenkoGSin20} for some cases of one-parameter families unfolding generally symmetric couples of heteroclinic and homoclinic tangencies. In principle, only two most important cases remain unproven: these are one-parameter families which unfold generally symmetric quadratic and cubic homoclinic tangencies. For both cases, the key problem is to find and study {\em symmetry-breaking bifurcations} in the first return maps.

One can standardly prove that these first return maps can be presented (in some rescaled coordinates and parameters)~\cite{GG09,DGG14,DGG15,GonGonOvs2017} as reversible maps close to the conservative quadratic and cubic H\'enon maps of the form~$\bar x = y,\; \bar y = M - x - y^2$ and~$\bar x = y,\; \bar y = - x + My \pm y^3$, respectively. These maps are reversible with respect to the involution~$h: x\to y,\; y\to x$, and their dynamics has been studied for a long time. However, until recently, their symmetry-breaking bifurcations have been completely unknown. As for the conservative quadratic H\'enon map, only in the very recent paper~\cite{GonchenkoGSaf20}, it has been shown that symmetry-breaking bifurcations occur starting with only 6-periodic orbits.

It is not the case of the cubic H\'enon maps \eqref{eq:HenonMapCubic}, where, as was shown in~\cite{DulinMeiss2000}, symmetry-breaking bifurcations can take place for 3-periodic orbits which emerge near the degenerate 1:3 resonance.
Moreover, in~\cite{DulinMeiss2000}, the corresponding bifurcation diagrams were constructed. In Sections 3 and 4, we complement these diagrams 
providing phase portraits  corresponding to transitions through bifurcation curves near the degenerate 1:3 resonance. Nevertheless, the main goal of the current paper is the study of peculiarities of \textit{reversible symmetry-breaking bifurcations} associated with the 1:3 resonance in the cubic H\'enon maps~\eqref{eq:HenonMapCubic} under reversible non-conservative perturbations, see Sections~\ref{sec:H_3plus}, \ref{sec:H_3minus} and~\ref{sec:md}.

Recall that for reversible systems, typical (codimension 1) local symmetry-breaking bifurcations are supercritical and subcritical pitchfork bifurcations~\cite{LermanTuraev2012}. As a result of a supercritical bifurcation, a symmetric elliptic periodic orbit bifurcates into a saddle and in its neighborhood a symmetric pair of stable and unstable periodic orbits emerges, see Figure~\ref{Fig:ReversiblePF}a-b. Under a subcritical bifurcation, the saddle periodic orbit becomes elliptic and there appears a symmetric pair of saddle periodic orbits, one with the Jacobian~$J<1$ and the other with the Jacobian~$J>1$, see Figure~\ref{Fig:ReversiblePF}c-d.
It is important to note that local symmetry-breaking bifurcations can be considered as an indicator of mixed dynamics in systems where the difference between intersecting attractor and repeller is invisible in standard numerics, for instance as in the nonholonomic model of rubber disk on the plane~\cite{GonchenkoGK20}.

\begin{figure}[t!]
\begin{minipage}[h]{1\linewidth}
\center{\includegraphics[width=0.8\columnwidth]{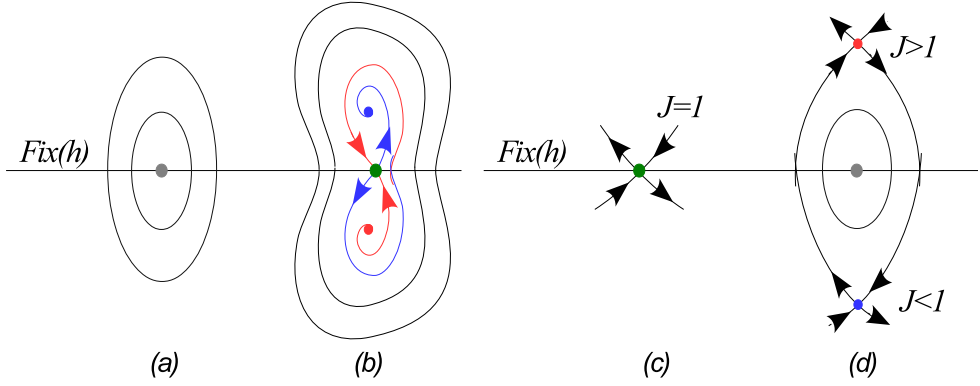}}
\end{minipage}
\caption{{\footnotesize Two types of reversible pitchfork bifurcation: (a)$\rightarrow$(b) supercritical pitchfork bifurcation at which an elliptic orbit (in plot~(a))  becomes a saddle orbit and a pair of stable and unstable orbits appears near the saddle (in plot~(b)); (c)$\rightarrow$(d) subcritical pitchfork bifurcation at which a saddle orbit (in plot~(c)) bifurcates into an elliptic orbit surrounded by a couple of saddle orbits with the Jacobians~$J>1$ and~$J<1$ (in plot~(d)).}}
		\label{Fig:ReversiblePF}
\end{figure}

Concerning global symmetry-breaking bifurcations, they are related to the appearance of non-transversal intersections between invariant manifolds of either the same periodic saddle orbit (homoclinic tangencies) or different saddles (heteroclinic tangencies). Under certain conditions bifurcations of these tangencies lead to the emergence of symmetric pairs of sinks and sources, area-expanding and area-contracting saddles as well as symmetric elliptic and saddle periodic orbits, and, hence, to the reversible mixed dynamics. Some of such global symmetry-breaking bifurcations were studied for non-transversal heteroclinic cycles of different types~\cite{GonShilTur97,LambStenkin2004,GonDelsh2013,GonchenkoGSin20}, see some examples of such cycles in Figure~\ref{Fig:HeteroclonicCycles}.	
	
\section{Construction of reversible non-conservative perturbations}\label{sec:constrRNP}
	
Following ideas of~\cite{GonchenkoGSaf20} we construct reversible non-conservative perturbations in the following conservative polynomial diffeomorphisms of the plane
$$
H_n: \;
\bar x = y,\;\;\;
\bar y = - x + P_n(y),
\label{eq:HenonMap}
$$
where~$P_n(y)$ is a polynomial of degree~$n$. These maps are also called \emph{H\'enon-like maps}. In particular, they include the cubic conservative H\'enon maps for~$P_3(y)=M_1+ M_2 y \pm y^3$.

The first approach to get reversible non-conservative perturbations is based on the fact that the second iteration of the inverse conservative cubic H\'enon maps~$H^{\pm}_3$ can be implicitly presented in the reversible cross-form
\begin{equation}
	f: (x,y)\to (\bar x, \bar y):  \;\;\; \bar x = F(x, \bar y),\;\;\;
	y = F(\bar y, x),
	\label{eq:RevMap}
\end{equation}
for which it is easy to add such perturbations. Indeed, the map~(\ref{eq:RevMap}) is reversible with respect to the involution 
$h: (x, y) \to (y, x)$, since 
$
h\circ f\circ h : \; \bar y = F(y,\bar x), \; x= F(\bar x, y),
$
and the inverse of this map coincides with~(\ref{eq:RevMap}). 
Thus, the map
$\bar y = F(x, \bar y) + \epsilon(x,\bar y),  y = F(\bar y,x) + \epsilon(\bar y,x)$  with an arbitrary smooth function $\epsilon$ is also reversible.
	
If~$F(x,y)$ in the right hand side of~(\ref{eq:RevMap}) is linear in~$x$, i.e.~$F(x,y) = x G_1(y) + G_0(y)$,
the map~(\ref{eq:RevMap}) can be rewritten in the explicit form
\begin{equation}
	%\begin{cases}
	f: \;\;\; \bar x = x G_1(\bar y) + G_0(\bar y), \;\;\;
	\bar y = \frac{y-G_0(x)}{G_1(x)}.
	%\end{cases}
	\label{eq:RevMapExplicit}
\end{equation}
%\begin{lemma}\label{lm:G}
If~$G_0$ and~$G_1$ are differentiable and~$G_1(y)\neq 0, y\in \mathbb{R}$, the map~(\ref{eq:RevMapExplicit}) is a diffeomorphism.	
%\end{lemma}
%\begin{proof}
Indeed, the Jacobian of~(\ref{eq:RevMapExplicit}) is
$
J(f)= {G_1(\bar y)}/{G_1(x)},
$
which is different from 0, since $G_1\neq 0$.
%\end{proof}

If we take~$F(x,y) = -x + P_n(y) + \varepsilon \varphi (x, y)$, where 
%$P_n(y)$ is a polynomial of degree~$n$, 
$\varepsilon$ is a small parameter and~$\varphi(x,y)$ is a function (preferably linear in~$x$) such that~(\ref{eq:RevMap}) is non-conservative, we get 
\begin{equation}
	\tilde{H}^2_n(\varepsilon)\;: \;\;\; \bar x =  - x + P_n(\bar y) + \varepsilon \varphi (x, \bar y), \;\;\;
	y =  - \bar y + P_n(x) + \varepsilon \varphi (\bar y, x),
	\label{eq:tldH2n}
\end{equation}
which is a reversible perturbation of the second iteration of the inverse H\'enon-like maps~$H^{-2}_n$. This follows from the following lemma.

\begin{lemma}\label{lm:H2n}
For~$\varepsilon=0$, the diffeomorphism~$\tilde{H}^2_n (\varepsilon)$, defined in~(\ref{eq:tldH2n}), satisfies the following property:
$$
		\tilde{H}^2_n(0)=(H_n)^{-1} \circ (H_n)^{-1}.
$$
\end{lemma}
	
\begin{proof}
We write the inverse map~$(H_n)^{-1}$ (solving the equations in \eqref{eq:HenonMap} for~$x$ and~$y$ and swapping~$\bar x \leftrightarrow x, \bar y \leftrightarrow y$) as
\[
(H_n)^{-1}\; : \;\;\; \bar x=-y +P_n(x), \;\;\; \bar y = x.
\]
The second iteration of this map
\[
(H_n)^{-2}=(H_n)^{-1} \circ (H_n)^{-1}\; : \;\;\; \bar x=-x +P_n(-y +P_n(x)), \;\;\; \bar y = -y +P_n(x)
\]
coincides with~$\tilde{H}^2_n(0)$.
\end{proof}
	
Recall that 
the bifurcation diagrams of~$H_n$ itself and the inverse map~$H_n^{-1}$ are related in the sense that an attractor of one map is a repeller of the other map and vice versa. Moreover,~$\tilde{H}^2_n(0)$ appears while studying bifurcations of heteroclinic cycles connecting a saddle with the Jacobian~$J>1$ and a saddle with~$J<1$, see Figure~\ref{Fig:HeteroclonicCycles}.
	
In the case of the cubic polynomials~$P_3(y)=M_1 + M_2 y \pm y^3$, the map~(\ref{eq:tldH2n}) is the perturbed second iteration of the inverse cubic conservative H\'enon map~$H^\pm_3$
\begin{equation}
	\tilde{H}_3^{\pm} (\varepsilon) \;:
	%\begin{cases}
	\bar x = M_1 + M_2 \bar y \pm \bar{y}^3 - x + \varepsilon \varphi (x, \bar y), \;\;\;
	y = M_1 + M_2 x \pm x^3 - \bar y + \varepsilon \varphi (\bar y, x).
	%\end{cases}
	\label{eq:HSqPerturb}
\end{equation}
In~\cite{SamylinaShyhmKazakov2017} a numerical study of bifurcations of 1:3 and 1:4 resonances was done for the perturbed system~$\tilde{H}^{2+}_3$ choosing~$\varphi(x,y)=x y$. For the convenience of the reader, we include these results in Section~\ref{sec:H_3plus}. We also provide the same analysis for the perturbation~$\tilde{H}^{2-}_3$ with~$\varphi(x,y)=x y$ in Section~\ref{sec:H_3minus}.
	
In the case of the map~(\ref{eq:tldH2n}), assuming the perturbation~$\varphi(x, y)=x \varphi_1 (y) + \varphi_0(y)$ is linear in the first variable~$x$, we can study the explicit map
$$
	\tilde{H}^2_n(\varepsilon)\;: \;\;\; \bar x =  - x + P_n(\bar y) + \varepsilon \varphi (x, \bar y), \;\;\;
	\bar y =  \frac{- y + P_n(x) +\varepsilon \varphi_0(x)}{1-\varepsilon \varphi_1 (x)}.
$$
Let us consider the perturbation~$\varphi(x,y)=x y$. In particular, the cubic map~$\tilde{H}^{2\pm}_3 (\varepsilon)$, given in~(\ref{eq:HSqPerturb}), turns out
\begin{equation}
	\tilde{H}^{2\pm}_3 (\varepsilon)\;:\;\;\;
	\bar x = M_1 + M_2 \bar y \pm \bar{y}^3 - x + \varepsilon x \bar y, \;\;\;
	\bar y = \frac{M_1 + M_2 x \pm x^3 - y}{1-\varepsilon x}.
	\label{eq:H2pl3expl}
\end{equation}
	
For~$\varepsilon=0$, the map~$\tilde{H}^{2\pm}_3 (0)$ is the composition~$({H}^{\pm}_3)^{-1}\circ ({H}^{\pm}_3)^{-1}$ due to Lemma~\ref{lm:H2n}, hence, it is reversible and conservative. Also the perturbed map~(\ref{eq:H2pl3expl}) is reversible with respect to the involution~$h: (x,y)\to (y,x)$ due to the cross-form presentation. However, for~$\varepsilon\neq 0$, the map ~$\tilde{H}^{2\pm}_3 (\varepsilon)$ is not conservative anymore since the Jacobian of the map is not equal to 1:
$$
	J = \frac{1-\varepsilon \bar y}{1 - \varepsilon x} \not\equiv 1.
$$
Therefore, one can consider the map~(\ref{eq:H2pl3expl}) as a non-conservative perturbation of~$({H}^{\pm}_3)^{-1}\circ ({H}^{\pm}_3)^{-1}$ that preserves the reversibility.
		
The second method consists in writing the perturbations in the form
\begin{equation}
	\tilde{H}_{n} (\varepsilon) \;: \;\;\; \bar x + \varepsilon \varphi(\bar y, \bar x) = y + \varepsilon \varphi(x,y), \;\;\; \bar y = -x + P_n(y+\varepsilon \varphi(x,y)),
	\label{eq:rev2}
\end{equation}
where~$\varepsilon$ is a small parameter and~$\varphi$ is a smooth function which gives a non-conservative perturbation. This perturbation can be obtained applying the so-called Quispel-Roberts method~\cite{QuispelRoberts92}, see also~\cite{GonchenkoGSaf20, BessaCR2015}. This method uses two facts: (i) any two-dimensional reversible map~$f$ can be presented as the composition of two involutions,~$f=\zeta_1\circ \zeta_2$, and a perturbed map is obtained perturbing one of the involutions,~$\tilde f=\zeta_1\circ \tilde\zeta_2$; (ii) if~$\zeta$ is an involution of~$f$ and the map~$T$ is a diffeomorphism, then~$\tilde\zeta = T^{-1}\circ \zeta \circ T$ is also involution of~$f$. Indeed,~$H_n=h\circ h_1$, where~$h: (x,y)\to (y,x)$,~$h_2: (x,y) \to (-x+P_n(y), y)$ are involutions. Thus,~$\tilde{H}_n = h\circ \tilde{h}_2$ is obtained perturbing the second involution~$\tilde{h}_2= \tilde h_2 = T^{-1}\circ h_2 \circ T$ with a near identity map~$T: \; \bar x = x, \bar y = y + \varepsilon \varphi(x,y)$.

\begin{lemma}
The diffeomorphism~$\tilde{H}_n(\varepsilon)$, defined in~(\ref{eq:rev2}), is reversible  with respect to the involution~$h: (x,y)\to (y,x)$.
\end{lemma}
	
\begin{proof} The proof is similar to the one done in~\cite{GonDelsh2013} for the map~(\ref{eq:RevMap}), see also~\cite{GonchenkoGSaf20}. %, see the sketch in Lemma~\ref{lm:frev}.
Indeed, to prove the reversibility of~$\tilde{H}_n(\varepsilon)$, we have to show that~$\tilde{H}_n(\varepsilon)^{-1}=h\circ \tilde{H}_n(\varepsilon)\circ h$.

First, we write the inverse map~$\tilde{H}_n(\varepsilon)^{-1}$ swapping the bar and no-bar variables~$\bar x \leftrightarrow x, \bar y \leftrightarrow y$:
\begin{equation}
		\tilde{H}_{n} (\varepsilon)^{-1} \;: \;\;\; \bar x = - y + P_n(x + \varepsilon \varphi(y, x)), \;\;\; \bar y + \varepsilon \varphi(\bar x, \bar y) =x + \varepsilon \varphi(y, x).
\label{eq:rev2inv}
\end{equation}
Second, the composition~$\tilde{H}_n(\varepsilon)\circ h$ is obtained interchanging the variables~$x\to y$,~$y\to x$ in~$\tilde{H}_n(\varepsilon)$ according to the involution~$h$
$$
		\tilde{H}_n(\varepsilon)\circ h\;: \;\;\; \bar x + \varepsilon \varphi(\bar y, \bar x) = x + \varepsilon \varphi(y,x), \;\;\; \bar y = -y + P_n(x+\varepsilon \varphi(y,x)).
$$
Then we apply~$h$ onto~$\tilde{H}_n(\varepsilon)\circ h$ swapping the bar variables~$\bar x \to \bar y$,~$\bar y \to \bar x$ and get
$
h\circ \tilde{H}_n(\varepsilon)\circ h %\;: \;\;\; \bar x = -y + P_n(x+\varepsilon \varphi(y,x)), \;\;\; \bar y + \varepsilon \varphi(\bar x, \bar y) = x + \varepsilon \varphi(y,x),
$
which coincides with~$\tilde{H}_n(\varepsilon)^{-1}$ in~(\ref{eq:rev2inv}). \end{proof}

For~$\varepsilon=0$, the map~(\ref{eq:rev2}) coincides with the conservative map~(\ref{eq:HenonMap}), not with the second iteration of the inverse map as in the first method.
We consider the concrete perturbation~$\varphi(x,y)= xy$
\begin{equation}
	\tilde{H}^{\pm}_{3}(\varepsilon) \;: \;\;\; \bar x + \varepsilon \bar x \bar y = y + \varepsilon xy, \;\;\; \bar y = -x + M_1 + M_2 (y+\varepsilon xy)  \pm (y+\varepsilon xy)^3.
	\label{eq:rev2cub}
\end{equation}
Since we choose a perturbation linear in the first variable~$x$, the equations in~(\ref{eq:rev2cub}) can be solved for~$\bar x$ and~$\bar y$. Thus, the maps~$\tilde{H}^{\pm}_{3}$ can be written in the explicit form. The Jacobian of the maps equals
$$
	J=\frac{1 + \varepsilon x}{1+\varepsilon \bar y}
$$
which is different from 1 for~$\varepsilon\neq 0$. We analyze  numerically the bifurcations of the 1:3 resonance in the perturbed cubic conservative H\'enon maps~$\tilde{H}^{d}_{3}(\varepsilon)$ for~$\varepsilon\neq 0$ and construct the corresponding bifurcation diagrams in Sections~\ref{sec:H_3plus} and~\ref{sec:H_3minus} paying special attention to the appearance of nonsymmetric %stable and completely unstable
periodic orbits under reversible pitchfork bifurcations.
	
\begin{remark}
The maps~(\ref{eq:H2pl3expl}) and~(\ref{eq:rev2cub}) are not diffeomorphisms in the whole plane~$\mathbf{R^2}$. However, the dynamics of this map is concentrated close to the origin, and for small~$\varepsilon$ the map can be considered as a diffeomorphism in a quite large neighborhood of the origin. Besides, for the first method, to have truly diffeomorphisms one can choose the perturbation~$\varphi (x,y)= x \arctan(y)$ instead of~$\varphi (x,y) = x y$.
\end{remark}

\section{1:3 resonance for the unperturbed maps~${H}^{\pm}_3$}\label{sec:1p3resUnpert}
	
In the general non-conservative setting, the analysis of the 1:3 resonance was done by Arnold~\cite{Arn-Geom}, see also~\cite{ArnKozNeisht07,Kuz}. For the study of the 1:3 resonance in conservative maps we refer to~\cite{BroerHeinsmann,SV09}. Recall that one can study the structure of such bifurcations writing the local normal form expressed in complex coordinates~$z=x+i y$ and~$z^*=x-i y$:
\begin{equation}
	\bar z = e^{ i 2\pi/3} (z + a_{02} (z^*)^2 + a_{21} z^2 z^*) + O(|z|^4),
	\label{eq:1p3nf}
\end{equation}
where the coefficient~$a_{02}$ is purely imaginary since map \eqref{eq:1p3nf} is reversible with respect to the involution~$h: z \to z^*$. In this case the 1:3 resonance is degenerate when~$a_{02}=0$.
	
Let~$d=\pm 1$ be the coefficient before the cubic term in~(\ref{eq:HenonMapCubic}), thus,~$d=1$ corresponds to~$H^+_3$ and~$d=-1$ stands for~$H^-_3$.
Then, for the parameters $M_1$ and $M_2$ in the 1:3 resonance curve
\begin{equation}
	l^d_{1:3} : M_1^2= \frac{d}{27} (1+M_2) (2M_2-7)^2,
	\label{eq:curveL13d}
\end{equation}
the map~${H}^{d}_3$ has the fixed point with eigenvalues~$e^{\pm i 2\pi/3}$ which is~$P_{1:3}^{(1)}=\sqrt{-d (1+M_2)/3} \left(1, 1\right)$  in the branch~$M_1>0$ and~$P_{1:3}^{(2)}=-\sqrt{-d (1+M_2)/3}\left(1, 1\right)$  in the branch~$M_1<0$. Note that the curve~$l^-_{1:3}$ has a self-intersection point at~$(M_1,M_2)=(0,7/2)$ where the map~$H^-_3$ has two fixed points~$P^{(1)}_{1:3}=\left(\sqrt{3/2}, \sqrt{3/2}\right)$ and~$P^{(2)}_{1:3}=\left(-\sqrt{3/2}, -\sqrt{3/2}\right)$ with eigenvalues~$e^{\pm i 2\pi/3}$ simultaneously.
	
The coefficients of the normal form~(\ref{eq:1p3nf}) are as follows:
$$
	\begin{array}{lll}
	a_{02}&=\displaystyle -2 i \sqrt{-d(1+M_2)}, &\text{ for } M_1>0,\\
	a_{02}&=\displaystyle 2 i \sqrt{-d(1+M_2)}, &  \text{ for } M_1<0,\\
	\displaystyle a_{21} &= %\displaystyle -4d (1+M_2)).
	-4d(1+M_2) + 4\sqrt{3} d M_2 i.
	&
	\end{array}
$$
It is easy to see that~$a_{02}$ and~$a_{21}$ do not vanish simultaneously. Moreover,~$a_{02}=0$ at~$M_2=-1$, therefore the 1:3 resonance is degenerate when~$M_1=0,M_2=-1$.
	
\begin{figure}[t]
\hspace{0.2cm}
\begin{minipage}{1.0\linewidth}		 \center{\includegraphics[width=1.0\columnwidth]{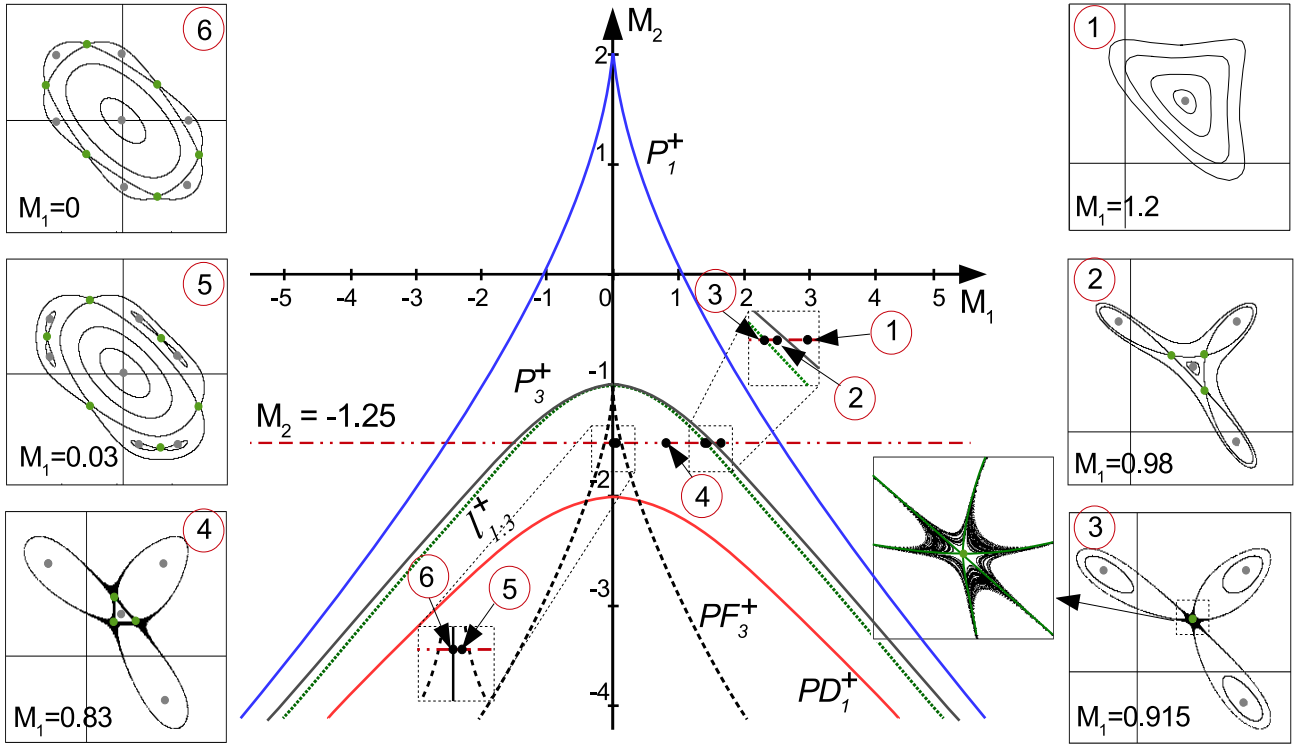}}
\end{minipage}
\caption{{\footnotesize Bifurcation diagram for the conservative cubic H\'enon map~$H_3^+$. The bifurcations curves~$P_1^+$, ~$l_{1:3}^+$, and~$PD_1^+$ are the curves of a parabolic (pitchfork for~$M_1=0$) bifurcation, 1:3 resonance, and a period-doubling bifurcation of fixed points, respectively. The curves~$P_3^+$ and~$PF_3^+$ are the curves of parabolic and pitchfork bifurcations of 3-periodic orbits, respectively. For the fixed value of~$M_2=-1.25$, the phase portraits are presented for~$M_1=~0, 0.03, 0.83, 0.915, 0.98$ and~$1.2$. For~$M_1<0$ the phase portraits are reflected symmetrically with respect to~$y=-x$.}}
	\label{Fig:HenonCubicBD_H3+}
\end{figure}
	
A detailed bifurcation analysis for the conservative cubic H\'enon maps~${H}^{\pm}_3$ was carried out in~\cite{DulinMeiss2000, GonGonOvs2017,GGOV18}. In particular, bifurcations of 3-periodic orbits were studied in~\cite{DulinMeiss2000}, and one of the principal bifurcations were pitchfork bifurcations.
For convenience, we display the corresponding bifurcation diagrams for~${H}^{+}_3$ and~${H}^{-}_3$ and complement them with the related phase portraits in Figures~\ref{Fig:HenonCubicBD_H3+} and~\ref{Fig:HenonCubicBD_H3-}, respectively.
	
\begin{figure}[t]
\hspace{0.2cm}
\begin{minipage}{1.0\linewidth}
 \center{\includegraphics[width=1\columnwidth]{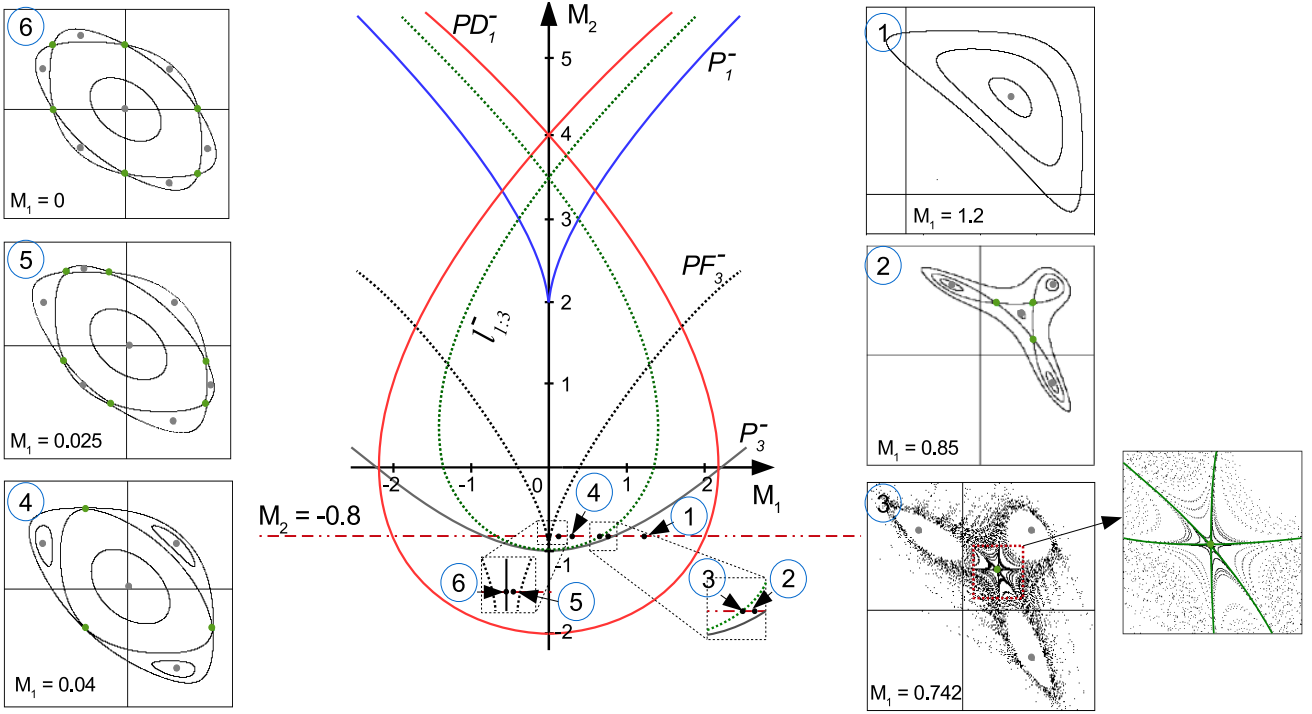}}
\end{minipage}
\caption{{\footnotesize Bifurcation diagram for the conservative cubic H\'enon map~$H_3^-$. The bifurcations curves~$P_1^-$, ~$l_{1:3}^-$, and~$PD_1^-$ are the curves of parabolic (pitchfork for~$M_1=0$) bifurcation, 1:3 resonance, and period-doubling bifurcation of fixed points, respectively. The curves~$P_3^-$ and~$PF_3^-$ are the curves of parabolic and pitchfork bifurcations of 3-periodic orbits, respectively. The sequence of phase portraits in the horizontal line~$M_2=-0.8$ is displayed for~$M_1=0, 0.025, 0.04, 0.742, 0.85$, and~$1.2$. For~$M<0$ the phase portraits are reflected symmetrically with respect to~$y=-x$.}}
		\label{Fig:HenonCubicBD_H3-}
\end{figure}
	
Let us briefly describe  these figures. Besides the 1:3 resonance curve~$l_{1:3}^d$, defined in~(\ref{eq:curveL13d}), the further bifurcation curves are~$P_1^d$,~$PD_1^d$,~$P_3^d$ and ~$PF_3^d$. The curves~$P_1^d$ and~$PD_1^d$, which have the following equations
$$
\begin{array}{ll}
	P_1^d :& \displaystyle M_1^2  = \frac{4d}{27} \left(2-M_2\right)^3\\
	PD_1^d : & \displaystyle M_1^2 = - \frac{4d}{27} \left(2+M_2\right)\left(4-M_2\right)^2, \\
\end{array}
$$
are the curves of fixed points with double eigenvalue (1,1) and (-1,-1), respectively. The curve~$P_1^d$ corresponds to a parabolic bifurcation of a fixed point of the map for~$M_1\neq 0, M_2\neq 2$. As a result of this bifurcation, crossing~$P_1^d$ laterally, there appear elliptic and saddle (hyperbolic) fixed points for parameters below~$P^+_1$ in the case of the map~$H_3^+$ and above~$P_1^-$ in the case of the map~$H_3^-$. When passing through the point~$M_1= 0, M_2= 2$ at~$P_1^d$ vertically (being~$M_1=0$ fixed) a subcritical and supercritical conservative pitchfork bifurcation takes place for~$H_3^+$ and~$H_3^-$, respectively. Namely, in the case of~$H_3^+$, under this bifurcation the saddle fixed point for~$M_2>2$ (above~$P^+_1$) becomes elliptic and a pair of saddle fixed points appears around for~$M_2<2$. For the map~$H_3^-$, passing through the point~$M_1= 0, M_2= 2$ from bottom to top, an elliptic fixed point undergoes a supercritical pitchfork bifurcation, it turns into a saddle fixed point and a pair of elliptic fixed points appears nearby.
The curve~$PD_1^d$ is related to a period-doubling bifurcation of a fixed point. At crossing~$PD_1^+$, the elliptic fixed point becomes saddle and in its neighborhood there appears an elliptic 2-periodic orbit. The curve~$PD_1^-$ corresponds for two different types of period-doubling bifurcation: the bottom part is responsible for a subcritical period-doubling bifurcation of saddle fixed point (which becomes an elliptic fixed point surrounded by a saddle 2-periodic orbit), while at the upper part of~$PD_1^-$, an elliptic fixed point undergoes a supercritical period-doubling bifurcation. See~\cite{GonGonOvs2017} for more details on these bifurcations. The other curves~$P_3^d$ and~$PF_3^d$ are associated with parabolic and conservative pitchfork bifurcations of 3-periodic orbits, respectively. They were discovered by~\cite{DulinMeiss2000} and they have too cumbersome expressions to be presented, so we omit their equations.
	
We note that the point~$(M_1,M_2)=(0,-1)$, corresponding to the case~$a_{02}=0$, is the cusp point of the curves~$PF_3^+$ and~$PF_3^-$ in both maps~$H^+_3$ and~$H^-_3$, respectively. This point also lies in~$P_3^d$ and~$l^d_{1:3}$. Getting inside the region bounded by the curve~$PF_3^+$ ($PF_3^-$), the symmetric elliptic (saddle) 3-periodic orbit becomes saddle (elliptic) and a pair of nonsymmetric elliptic (saddle) orbits of the same period emerges for~$H^+_3$ ($H^-_3$). It is worth mentioning that in the nonconservative reversible case, instead of nonsymmetric elliptic (saddle) orbits, a pair of stable and unstable orbits (saddles with the Jacobians~$J>1$ and~$J<1$) emerges as a result of a reversible pitchfork bifurcation.
	
In Figure~\ref{Fig:HenonCubicBD_H3+} we present a sequence of phase portraits near the 1:3 resonance for a fixed~$M_2$ for~$H^+_3$. Let us give some details of the bifurcations which take place in the horizontal line~$M_2=-1.25$. For~$M_1$ on the right hand side of the right branch of~$P^+_3$ (see phase portrait~\raisebox{.5pt}{\textcircled{\raisebox{-.9pt} {1}}} for~$M_1=1.2$), there is an elliptic fixed point~$E_1$ which is born  along with a saddle fixed point (the latter is not presented in the phase portrait) after a parabolic bifurcation at~$P^+_1$. At crossing~$P^+_3$ a parabolic bifurcation of 3-periodic orbits occurs and close to the elliptic point~$E_1$ there appear 3-periodic orbits~$E_3$ and~$S_3$ of elliptic  and saddle  type, respectively, in the right hand side of~$P_3^+$ (see, for example, phase portrait~\raisebox{.5pt}{\textcircled{\raisebox{-1.5pt} {2}}} for~$M_1=0.98$). In the 1:3 resonance curve~$l_{1:3}$, the saddle 3-periodic orbit~$S_3$ collides with the elliptic point~$E_1$ and they become the saddle fixed point~$P^{(1)}_{1:3}$ with 6 separatrices (see phase portrait~\raisebox{.5pt}{\textcircled{\raisebox{-.9pt} {3}}} for~$M_1\approx 0.915$). After this bifurcation, the saddle 3-periodic orbit~$S_3$ is reconstructed, and the homoclinic connections are transformed into heteroclinic cycles (see phase portrait~\raisebox{.5pt}{\textcircled{\raisebox{-.9pt} {4}}} for~$M_1=0.83$). At passage through the left curve~$PF^+_3$ the elliptic 3-periodic orbit~$E_3$ undergoes a supercritical pitchfork bifurcation: it becomes saddle~$\hat{S}_3$  and a pair of nonsymmetric elliptic 3-periodic orbits~$\hat{E}^{1}_3$ and~$\hat{E}^2_3$ appears nearby (see, for instance, phase portrait~\raisebox{.5pt}{\textcircled{\raisebox{-1.5pt} {5}}} for~$M_1=0.03$). Then the elliptic orbits~$\hat{E}^{1}_3$ and~$\hat{E}^2_3$ move away from~$\hat{S}_3$ and get closer to~$S_3$. At the same time, the separatrices of interior saddle~$\hat{S}_3$ increase until they connect with the separatrices of the exterior saddle~$S_3$ (the phenomenon of splitting of separatrices takes place). After some bifurcation related with homo/heteroclinic connections (phase portrait~\raisebox{.5pt}{\textcircled{\raisebox{-.9pt} {6}}} at~$M_1= 0$) the saddle and elliptic 3-periodic orbits are rotated. Afterwards, by the symmetry in the~$(M_1,M_2)$-plane, for~$M_1>0$, the periodic orbits undergo an inverse pitchfork bifurcation (the elliptic 3-periodic orbits~$\hat{E}^{1}_3$ and~$\hat{E}^2_3$ merge into the saddle orbit~$S_3$ which becomes elliptic~$E_3$) at~$PF^+_3$; a 1:3 resonance bifurcation (the saddle 3-periodic orbit~$\hat{S}_3$ is reconstructed after passing through the saddle fixed point~$P^{(2)}_{1:3}$ with 6 separatrices) takes place at~$l^+_{1:3}$; an inverse parabolic bifurcation (the saddle and elliptic 3-periodic orbits~$E_3$ and~$\hat{S}_3$ merge into the elliptic fixed point~$E_1$) occurs at crossing~$P^+_3$.
	
In Figure~\ref{Fig:HenonCubicBD_H3-} one can observe the bifurcations which occur at crossing the curves~$l^-_{1:3}$, $P^-_3$ and~$PF^-_3$ in the case of the map~$H_3^-$. We consider the horizontal line~$M_2=-0.8$. For~$M_1$ in the right hand side of~$P^-_3$ (see, for example, phase portrait~\raisebox{.5pt}{\textcircled{\raisebox{-.9pt} {1}}} for~$M_1=1.2$), there are an elliptic fixed point~$E_1$ and a saddle 2-periodic orbit (the latter is absent in the figure) which appear after a period-doubling bifurcation at~$PD^-_1$. Decreasing~$M_1$ and passing through the curve~$P^-_3$, elliptic and saddle 3-periodic orbits~$E_3$ and~$S_3$ show up surrounding the elliptic fixed point~$E_1$ (see, for instance, phase portrait~\raisebox{.5pt}{\textcircled{\raisebox{-.9pt} {2}}} for~$M_2=0.85$). Further, for~$M_1$ in the curve~$l_{1:3}^-$ the elliptic point~$E_1$ and the saddle orbit~$S_3$ are transformed into the saddle point~$P^{(1)}_{1:3}$ with 6 separatrices (see phase portrait~\raisebox{.5pt}{\textcircled{\raisebox{-.9pt} {3}}} at~$M_1\approx 0.74245$), and after crossing~$L_{1:3}^-$ the saddle 3-periodic orbit~$S_3$ is rotated reconstructing the homoclinic configuration into the heteroclinic connections (see phase portrait~\raisebox{.5pt}{\textcircled{\raisebox{-.9pt} {4}}} for~$M_1=0.04$). After that at crossing the right branch of~$PF^-_3$, the saddle 3-periodic orbit~$S_3$ goes through a subcritical pitchfork bifurcation: the saddle orbit becomes elliptic~$\hat{E}_3$ and in its neighborhood there appear two nonsymmetric saddle 3-periodic orbits~$\hat{S}^{1}_3$ and~$\hat{S}^2_3$ (as in phase portrait~\raisebox{.5pt}{\textcircled{\raisebox{-.9pt} {5}}} at~$M_1=0.025$). Varying further~$M_1$ the nonsymmetric saddle orbits~$\hat{S}^{1}_3$ and~$\hat{S}^2_3$ move away from the elliptic orbit~$\hat{E}_3$  toward the other elliptic orbit~$E_3$ (see, for instance, phase portrait~\raisebox{.5pt}{\textcircled{\raisebox{-.9pt} {6}}} for~$M_1= 0$). For~$M_1<0$, the two saddle orbits~$\hat{S}^{1}_3$ and~$\hat{S}^2_3$ get closer to~$E_3$. These three orbits undergo an inverse pitchfork bifurcation while crossing the left branch of~$PF^-_3$: the two saddle and elliptic 3-periodic orbits~$\hat{S}^{1}_3$,~$\hat{S}^2_3$ and~$\hat{E}_3$ collide into the saddle 3-periodic orbit~$S_3$. At crossing the left branch of~$l^-_{1:3}$, the rotation of the saddle 3-periodic orbit~$S_3$ takes place. Finally, for the parameters in the curve~$P^-_3$ the 3-periodic orbits~$\hat{E}_3$ and~$S_3$ disappear and the elliptic fixed point~$E_1$ remains.
	
\begin{remark}
The phase portraits in Figures~\ref{Fig:HenonCubicBD_H3+} and~\ref{Fig:HenonCubicBD_H3-} for values of parameters ($M_1,M_2$) and ($-M_1,M_2$) are central symmetric with respect to the origin~$x=y=0$. This is due to the fact that the maps~(\ref{eq:HenonMapCubic}) are invariant under the change~$x\to -x, y\to -y, M_1\to -M_1$. This also results in the symmetries of the bifurcation curves in the~$(M_1,M_2)$ parameter plane with respect to~$M_1=0$.
\label{rem2}
\end{remark}
	
\begin{remark}
In Figures~\ref{Fig:HenonCubicBD_H3+} and~\ref{Fig:HenonCubicBD_H3-}, the saddle separatrices are shown coinciding for simplicity, although they are not expected to be exactly merged since the phenomenon of splitting of separatrices occurs (see, for instance,~\cite{DelshamsGG16,DelshamsGG20} and references therein for more details about this phenomenon).
\end{remark}

\section{On the degeneracy of the~$p:q$ resonances with odd~$q>3$ in~$H^\pm_3$ with~$M_1=0$.}

Let us consider the conservative cubic H\'enon maps with~$M_1=0$. Due to Remark~\ref{rem2}, there is the central symmetry in the phase portraits~$x\to -x, y\to -y$. A normal form for the~$p:q$ resonance, where~$p$ and~$q$ are mutually prime and~$q > 3$ is odd, is as follows
$$
\bar z = e^{i 2\pi p/q} (z+\Omega(|z|^2)z^*+A(z^*)^{q-1} + Bz^{q+1} + Cz(z^*)^q).
$$
The corresponding flow normal form in this case can be written as follows:
\begin{equation}
\dot z = iz + \Omega(|z|^2)z^* + A(z^*)^{q-1} + Bz^{q+1} + Cz(z^*)^q.
\label{eq:pqnf}
\end{equation}
It is conservative and reversible with respect to the involution~$(t, z) \to (-t, z^*)$. Applying the involution gives
$$
- \dot z^* =i z^* + \Omega(|z|^2)z + A(z)^{q-1} + B(z^*)^{q+1} + Cz^*(z)^q.
$$
Further, let us consider the complex conjugate system
$$
- \dot z =-i z + \Omega^*(|z|^2)z^*+ A^*(z^*)^{q-1} + B^*(z)^{q+1} + C^*z(z^*)^q.
$$
Thus, the reversibility implies that~$\Omega = - \Omega^*, A=-A^*, B=-B^*, C=-C^*$, i.e. all coefficients in \eqref{eq:pqnf} should be pure imaginary. Therefore, equation \eqref{eq:pqnf} takes the form
\begin{equation}
\dot z = iz + i\Omega(|z|^2)z^* + iA(z^*)^{q-1} + iBz^{q+1} + iCz(z^*)^q,           
\label{eq:pqnf2}
\end{equation}
where all coefficients~$\Omega, A, B, C$ are real. The conservativity condition means zero divergence, i.e.
$$
\frac{\partial \dot z}{\partial z} + \frac{\partial \dot z^*}{\partial z^*} \equiv 0
$$
As it follows from \eqref{eq:pqnf2}
$$
\frac{\partial \dot z}{\partial z} =  i + i\Omega^\prime (z^*)^2 + i(q+1)Bz^q + iC(z^*)^q
$$
and
$$
\frac{\partial \dot z^*}{\partial z^*} = - i  - i \Omega^\prime (z)^2  -i (q+1)B(z^*)^q  - i C(z)^q
$$
Thus, the conservativity condition is
$$
C+B(q+1)=0.
$$
The symmetry~$z\to -z$ implies~$A=B=C=0$ since~$q$ is odd. Thus, the above condition is automatically fulfilled, and the following result holds.

\begin{lemma}\label{lm:pq}
For maps~$H_3^{\pm}$ with~$M_1=0$, any~$p:q$ resonance at the fixed point~$O(0,0)$, where~$q> 3$ is odd, is at least triple degenerate.
\end{lemma}

In Figure~\ref{fig:FigPQRes} we illustrate this result for both cubic H\'enon maps~$H_3^{\pm}$. In Figures~\ref{fig:FigPQRes}a,b we show phase portraits near the degenerated 1:5 and 1:7 resonances for the map~$H_3^{+}$. These resonances occur at~$M_2 \approx 0.575$ and~$M_2 \approx 1.15$, respectively. As a result, four periodic orbits emerge: a pair of symmetric periodic saddles (colored in light and dark green, respectively) and a pair of nonsymmetric periodic elliptic orbits (colored in grey and black, respectively). In Figure~\ref{fig:FigPQRes}c,d we demonstrate phase portraits for the map~$H_3^{-}$. In contrast to the previous case, here periodic elliptic orbits are symmetric while periodic saddles are nonsymmetric. Here the 1:5 resonance occurs at~$M_2 \approx 0.66$ and the 1:7 resonance takes place at~$M_2 \approx 1.36$.

\begin{figure}[tbh]
\center{\includegraphics[width=0.8\linewidth]{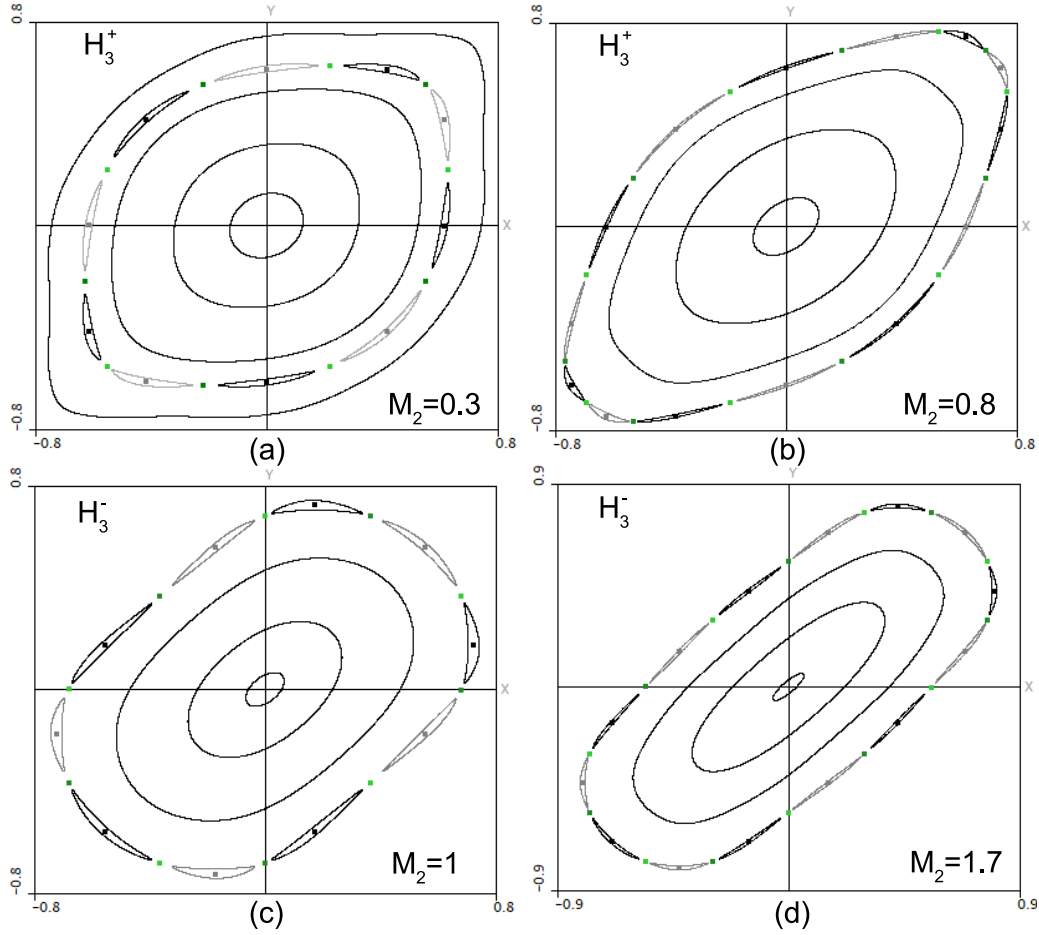} }
\caption{{\footnotesize Phase portraits near the degenerated 1:5 (left column) and 1:7 (right column) resonances in the conservative cubic H\'enon maps~$H_3^{+}$ (top row) and~$H_3^{-}$ (bottom row).}}
\label{fig:FigPQRes}
\end{figure}

\section{On the 1:3 resonance for the perturbed maps~$\tilde{H}^{+}_3(\varepsilon)$ and~$\tilde{H}^{2+}_3$.}\label{sec:H_3plus}
	
In this section we describe symmetry-breaking bifurcations near the 1:3 resonance in~$\tilde{H}^{+}_3(\varepsilon)$ in the form~(\ref{eq:rev2cub}). Note that the bifurcation picture is qualitatively similar for~$\tilde{H}^{2+}_3(\varepsilon)$, defined in~(\ref{eq:H2pl3expl}).
	
We apply reversible non-conservative perturbations to the conservative map~$H^+_3$ and study their impact on the structure of the 1:3 resonance. %in~$\tilde{H}^{+}_3 (\varepsilon)$ given in~(\ref{eq:rev2cub}).  
We display the bifurcation diagram for the fixed perturbation parameter~$\varepsilon=0.05$ in Figure~\ref{Fig:HenonPerturb13}. In comparison to the unperturbed case, we can see the slightly changed bifurcation curves~$l_{1:3}^+$, $P_3^+$ and~$PF_3^+$, related to the 1:3 resonance, a parabolic bifurcation of the appearance of 3-periodic orbits and a reversible pitchfork bifurcation of 3-periodic orbits, respectively. Unlike the conservative case, the bifurcation curves are not symmetric since the invariance of the map with respect to the change~$M_1\to -M_1$ (see also Remark~\ref{rem2}) is not conserved anymore at adding the perturbation. However, the symmetry in the phase portraits with respect to the straight line~$y=x$ is preserved due to the reversibility.  Also the curve~$PF_3^+$ here is associated with the symmetry-breaking bifurcations which are non-conservative reversible pitchforks.

\begin{figure}[t]
\hspace{0.2cm}
\begin{minipage}{1\linewidth}
\center{\includegraphics[width=0.9\columnwidth]{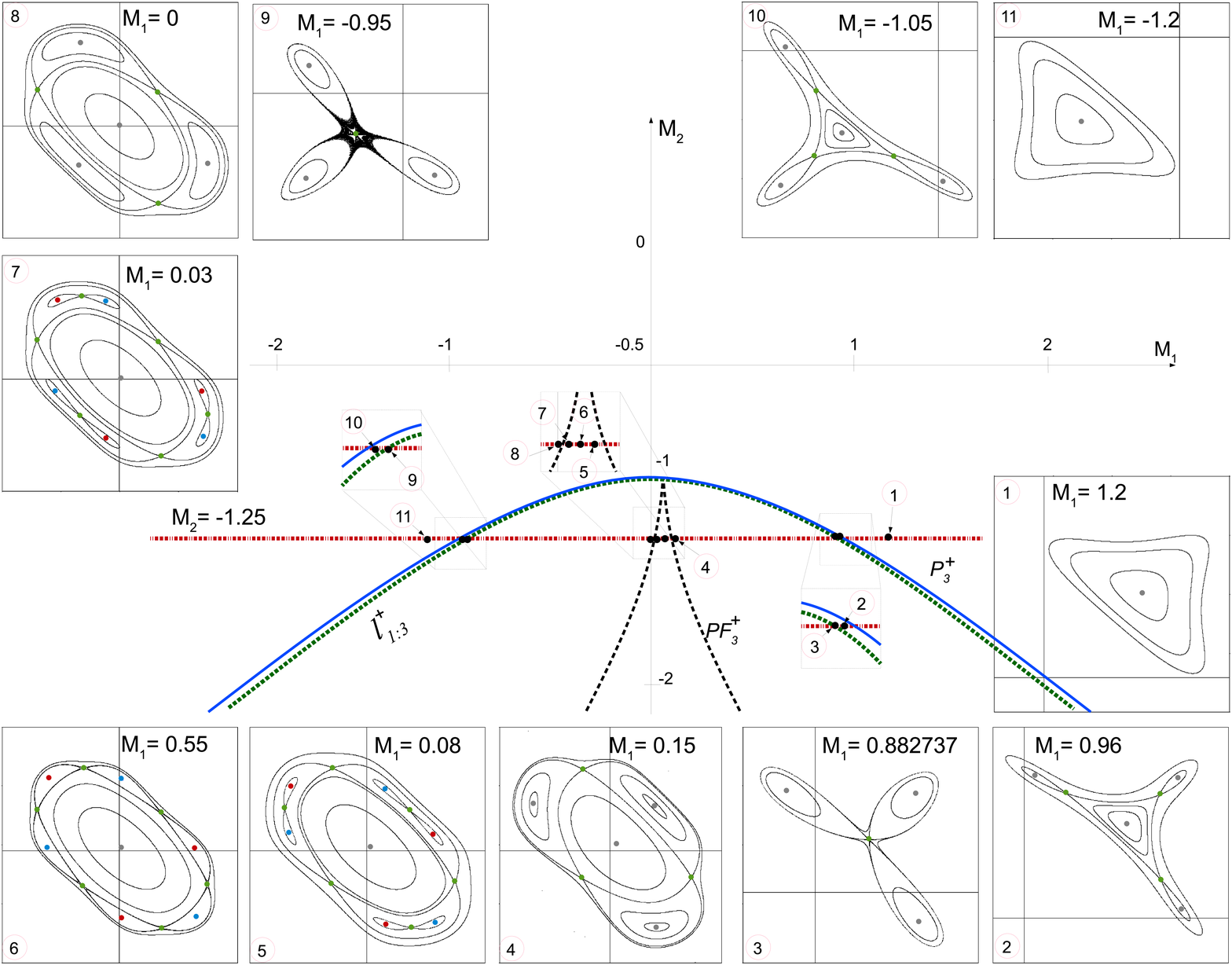}}
\end{minipage}
\caption{
Bifurcation diagram for the reversible non-conservative map~$\tilde{H}_3^+(\varepsilon)$  for~$\varepsilon = 0.05$ near the 1:3 resonance. The bifurcations curve~$l_{1:3}^+$ is the curve of the 1:3 resonance of fixed points, while the curves~$P_3^+$ and~$PF_3^+$ are related to parabolic and reversible pitchfork (non-conservative symmetry-breaking) bifurcations of 3-periodic orbits, respectively. For the fixed value of~$M_2=-1.25$, the phase portraits are given for~$M_1=-1.2, -1.05, -.0.95, 0, 0.003, 0.055, 0.08, 0.15, 0.8827, 0.96$ and~$1.2$.  The green, grey, red and blue stand for saddle, elliptic, stable (sinks) and unstable (sources) fixed points and periodic orbits, respectively.
}
\label{Fig:HenonPerturb13}
\end{figure}
	
Let us describe the sequence of bifurcations which occur for the fixed parameter~$M_2=-1.25$ and decreasing the parameter~$M_1$ in Figure~\ref{Fig:HenonPerturb13}. To the right hand side of the curve~$P_{3}^+$ (see, for example, the phase portrait~\raisebox{.5pt}{\textcircled{\raisebox{-.9pt} {1}}}  at~$M_1=1.2$), the map~$\tilde{H}_3^+(\varepsilon)$ (and also~$\tilde{H}^{2+}_3$) has an elliptic fixed point~$E_1$.  At crossing~$P_3^+$, there appear a symmetric elliptic 3-periodic orbit~$E_3$ and a symmetric saddle 3-periodic orbit~$S_3$ with homoclinic loops (see the phase portrait~\raisebox{.5pt}{\textcircled{\raisebox{-.9pt} {2}}} for the parameter value~$M_1=0.96$). Then approaching the curve~$l_{1:3}^+$, the orbits~$S_3$ and~$E_1$  merge into a saddle fixed point with 6 separatrices (phase portrait~\raisebox{.5pt}{\textcircled{\raisebox{-.9pt} {3}}} at~$M_1\approx 0.882737$). After that the saddle splits into elliptic point~$E_1$ and saddle 3-periodic orbit~$S_3$, now~$S_3$ is rotated and form heteroclinic connections (phase portrait~\raisebox{.5pt}{\textcircled{\raisebox{-.9pt} {4}}} for~$M_1=0.15$). Further, the elliptic 3-periodic orbit~$E_3$ undergoes a reversible pitchfork bifurcation at crossing the right branch of the curve~$PF_3^+$: the elliptic 3-periodic orbit~$E_3$ breaks into saddle, stable and unstable 3-periodic orbits~$\hat{S}_3$,~$\hat{A}_3$ and~$\hat{R}_3$, respectively (phase portrait~\raisebox{.5pt}{\textcircled{\raisebox{-.9pt} {5}}} for~$M_1=0.08$). Note that under this bifurcation a pair of nonsymmetric 3-periodic orbits~$\hat{A}_3$ and~$\hat{R}_3$ is born. The separatrices of each component of the new saddle~$\hat{S}_3$ tend to the the corresponding components of the stable and unstable orbits~$\hat{A}_3$ and~$\hat{R}_3$ at forward and backward iterations, respectively, forming homoclinic loops, and all the three components are surrounded by invariant curves in the similar way as in Figure~\ref{Fig:ReversiblePF}(a). Decreasing~$M_1$, the stable and unstable orbits~$\hat{A}_3$ and~$\hat{R}_3$ move away from each other. Moreover, one of the components of the stable (or unstable) 3-periodic orbit gets away from the symmetry line~$y=x$, while the other two components move closer each other and to the line~$y=x$. At the same time, the separatrices of the inner saddles become larger. At some moment (close to~$M_1=0.055$, see phase portrait~\raisebox{.5pt}{\textcircled{\raisebox{-.9pt} {6}}}), the separatrices of the inner and exterior saddles~$\hat{S}_3$ and~$S_3$ merge (not exactly due to splitting of separatrices) and the transformation of homoclinic/heteroclinic connections takes place after which all the involved 3-periodic orbits are rotated (as in phase portrait~\raisebox{.5pt}{\textcircled{\raisebox{-1.5pt} {7}}} for~$M_1=0.03$). At crossing the left branch of~$PF_3^+$ an inverse pitchfork bifurcation occurs: the saddle, stable and unstable 3-periodic orbits~$S_3$,~$\hat{A}_3$ and~$\hat{R}_3$ merge into an elliptic 3-periodic orbit~$E_3$ (phase portrait~\raisebox{.5pt}{\textcircled{\raisebox{-.9pt} {8}}} at~$M_1=0$). The remaining saddle  3-periodic orbit~$\hat{S}_3$ and elliptic point~$E_1$ collide at crossing the left branch of~$l_{1:3}^+$, there is a saddle point with 6 separatrices for the parameters in this curve (for a parameter close to~$M_1=-0.95$, see phase portrait~\raisebox{.5pt}{\textcircled{\raisebox{-.9pt} {9}}}). After the bifurcation, the 6-separatrix saddle splits into saddle and elliptic 3-periodic orbits~$\hat{S}_3$ and~$E_3$ in the left hand side of~$l^+_{1:3}$, now~$\hat{S}_3$ is rotated by~$\pi/3$ and the heteroclinic cycles change into homoclinic loops (phase portrait~\raisebox{.5pt}{\textcircled{\raisebox{-.9pt} {\scriptsize 10}}}  for~$M_1=-1.05$). Finally, we transit the curve~$P_3^+$ and the 3-periodic orbits~$\hat{S}_3$ and~$E_3$ disappear (phase portrait~\raisebox{.5pt}{\textcircled{\raisebox{-.9pt} {\scriptsize 11}}} for~$M_1=-1.2$).
	
Thus, at transition into the domain lying below the curve~$PF_3^+$, there appears a symmetric pair of nonsymmetric stable and completely unstable 3-periodic orbits. This fact is relevant for detecting mixed dynamics in maps with symmetric cubic homoclinic tangencies whose truncated first return map is~$\tilde{H}^+_3(\varepsilon)$ or~$\tilde{H}^{2+}_3(\varepsilon)$. Thus, there are Newhouse domains where maps with infinitely many attracting, repelling, saddle and elliptic periodic orbits are dense.
		
\section{On the 1:3 resonance in the perturbed maps~$\tilde{H}^{-}_3(\varepsilon)$ and~$\tilde{H}^{2-}_3(\varepsilon)$}\label{sec:H_3minus}
	
In this section we study how the 1:3 resonance evolves under the perturbation in the case of the map~$\tilde{H}^{-}_3(\varepsilon)$. The corresponding bifurcation diagram is illustrated in Figure~\ref{Fig:HenonCubicBD_tldH3-}. Unlike the unperturbed case in Figure~\ref{Fig:HenonCubicBD_H3-}, the 1:3 resonance curve~$l_{1:3}^-$ as well as the curves~$P_3^-$ and~$PF_3^-$ related to parabolic and reversible pitchfork bifurcations of 3-periodic orbits, respectively, are nonsymmetric and slightly moved, since the symmetry with respect to~$M_1=0$ (see also Remark~\ref{rem2}) is not conserved anymore. Also~$PF_3^-$ corresponds to a non-conservative symmetry-breaking bifurcation after which two nonsymmetric saddle 3-periodic orbits appear, one of them with the Jacobian~$J> 1$ and the other with the Jacobian~$J< 1$. Note that for~$\tilde{H}^{2-}_3(\varepsilon)$, the bifurcation diagram and bifurcation sequence of phase portraits is qualitatively the same.
	
\begin{figure}[t]
\hspace{0.2cm}
\begin{minipage}{1\linewidth}			 \center{\includegraphics[width=0.9\columnwidth]{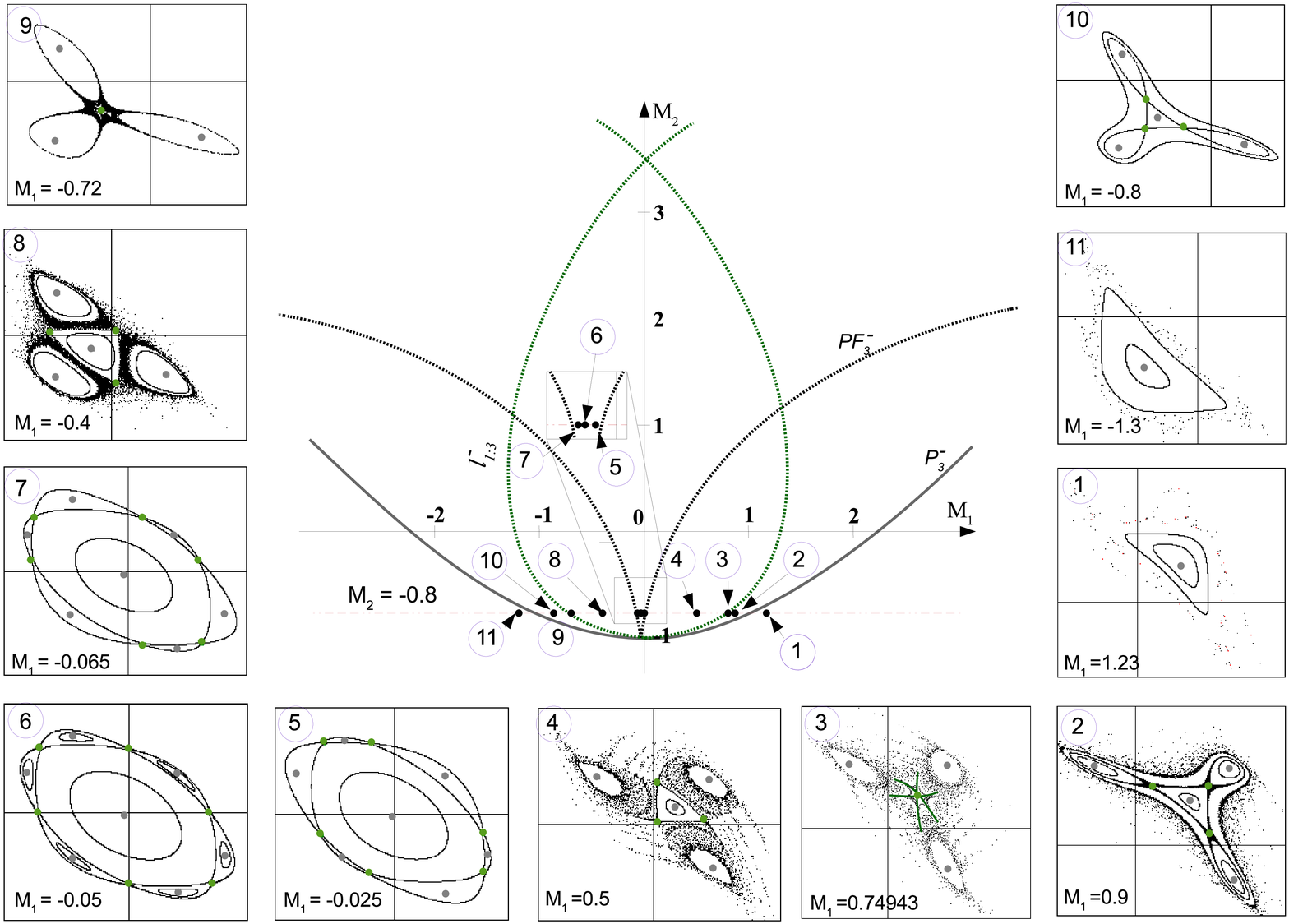}}
\end{minipage}
\caption{{
Bifurcation diagram for the the reversible non-conservative map~$\tilde{H}_3^-(\varepsilon)$  for~$\varepsilon = 0.05$ near the 1:3 resonance. The bifurcations curve~$l_{1:3}^-$ is related to bifurcations of the 1:3 resonance of fixed points, while the curves~$P_3^-$ and~$PF_3^-$ correspond to parabolic and reversible pitchfork (non-conservative symmetry-breaking) bifurcations of 3-periodic orbits, respectively. For the fixed value of~$M_2=-0.8$, the phase portraits are given for~$M_1=-1.3, -0.8, -0.72, -0.4, -0.065, -0.005, -0.025, 0.5, 0.7494, 0.9$ and~$1.23$.  The green and grey points stand for saddle and elliptic orbits, respectively.  		
}}
\label{Fig:HenonCubicBD_tldH3-}
\end{figure}
	
Let us give details on the bifurcations taking place in the bifurcation diagrams in Figure~\ref{Fig:HenonCubicBD_tldH3-}. We choose the horizontal line~$M_2=-0.8$. We start with  the parameters in the right hand side of~$P_3^-$ (see, for instance, phase portrait~\raisebox{.5pt}{\textcircled{\raisebox{-.9pt} {1}}} for~$M_1=1.23$), where the map~$\tilde{H}^{-}_3(\varepsilon)$ (and~$\tilde{H}^{2-}_3(\varepsilon)$) has an elliptic fixed point~$E_1$. The point~$E_1$ undergoes a parabolic bifurcation at crossing~$P_3^-$ and there appear 3-periodic orbits~$S_3$ and~$E_3$ of saddle  and elliptic  type close to~$E_1$ for the parameters in the left hand side of~$P_3^-$ (phase portrait~\raisebox{.5pt}{\textcircled{\raisebox{-.9pt} {2}}} for~$M_1=0.9$). Note that the three components of the saddle orbit~$S_3$ have homoclinic loops. In the 1:3 resonance curve~$l_{1:3}^-$ (at~$M_1\approx 0.74943$), the saddle 3-periodic orbit~$S_3$ and the elliptic point~$E_1$ merge into a saddle fixed point with 6 separatrices (see phase portrait~\raisebox{.5pt}{\textcircled{\raisebox{-.9pt} {3}}}) which for the parameters in the left hand side of~$l_{1:3}^-$ breaks into saddle and elliptic 3-periodic orbits~$S_3$ and~$E_3$  again, but in the left hand side of~$l_{1:3}^-$ the orbit~$S_3$ is rotated by~$\pi/3$ and the homoclinic loops of~$S_3$ are reorganized into heteroclinic connections (see, for example, phase portrait~\raisebox{.5pt}{\textcircled{\raisebox{-.9pt} {4}}}  at~$M_1=0.5$). Crossing the right branch of the curve~$PF_3^-$, the saddle orbit~$S_3$ undergoes the subcritical pitchfork bifurcation. As a result, the saddle orbit is converted into an elliptic 3-periodic orbit~$\hat{E}_3$ and two saddle 3-periodic orbits~$\hat{S}_3^1$ and~$\hat{S}_3^2$ show up nearby (see the phase portrait~\raisebox{.5pt}{\textcircled{\raisebox{-.9pt} {5}}} for~$M_1=-0.025$). Moreover, the Jacobian in~$\hat{S}_3^1$ is greater than 1 and the Jacobian in~$\hat{S}_3^2$ is less than 1. Decreasing the parameter~$M_1$, the saddle orbits~$\hat{S}_3^1$ and~$\hat{S}_3^2$ distance from each other and they move toward the elliptic orbit~$E_3$ (phase portraits~\raisebox{.5pt}{\textcircled{\raisebox{-.9pt} {6}}}  and~\raisebox{.5pt}{\textcircled{\raisebox{-.9pt} {7}}}  for~$M_1=-0.05$ and~$M_1=-0.065$). In the left branch of~$PF_3^-$, an inverse pitchfork bifurcation takes place: the orbit~$E_3$ merges along with~$\hat{S}_3^1$ and~$\hat{S}_3^2$ into a saddle 3-periodic orbit~$S_3$ (phase portrait~\raisebox{.5pt}{\textcircled{\raisebox{-.9pt} {8}}} for~$M_1=-0.4$). Afterwards, the reconstruction of~$S_3$ happens in the 1:3 resonance curve~$l_{1:3}^-$ (phase portrait~\raisebox{.5pt}{\textcircled{\raisebox{-.9pt} {9}}} and~\raisebox{.5pt}{\textcircled{\raisebox{-.5pt} {\scriptsize 10}}} for~$M_1=-0.72$ and~$M_1=-0.8$). Finally, the 3-periodic orbits~$\hat{E}_3$ and~$S_3$ disappear at crossing~$P_3^-$ (phase portrait~\raisebox{.5pt}{\textcircled{\raisebox{-.5pt} {\scriptsize 11}}}) for~$M_1=-1.3$).
	
Note that in the domain above the curve~$PF_3^-$ there emerges a symmetric pair of nonsymmetric saddle 3-periodic orbits whose Jacobians are greater and less than 1. This configuration also implies the existence of mixed dynamics in maps with cubic homoclinic tangencies whose truncated first return map is~$\tilde{H}^{-}_3(\varepsilon)$ or~$\tilde{H}^{2-}_3(\varepsilon)$.  Also in~$\tilde{H}^-_3(\varepsilon)$ itself we show numerically  the existence of mixed dynamics for the parameters from the domain inside~$PF_3^-$ since heteroclinic connections between two saddle orbits, one with the Jacobian~$J>1$ and the other with the Jacobian~$J<1$, leads to the presence of the so-called Lamb-Stenkin non-transversal heteroclinic cycle~\cite{LambStenkin2004}. See more details in Section~\ref{sec:md}.

\begin{figure}[t!]
\begin{minipage}[h]{1\linewidth}
\center{\includegraphics[width=0.9\columnwidth]{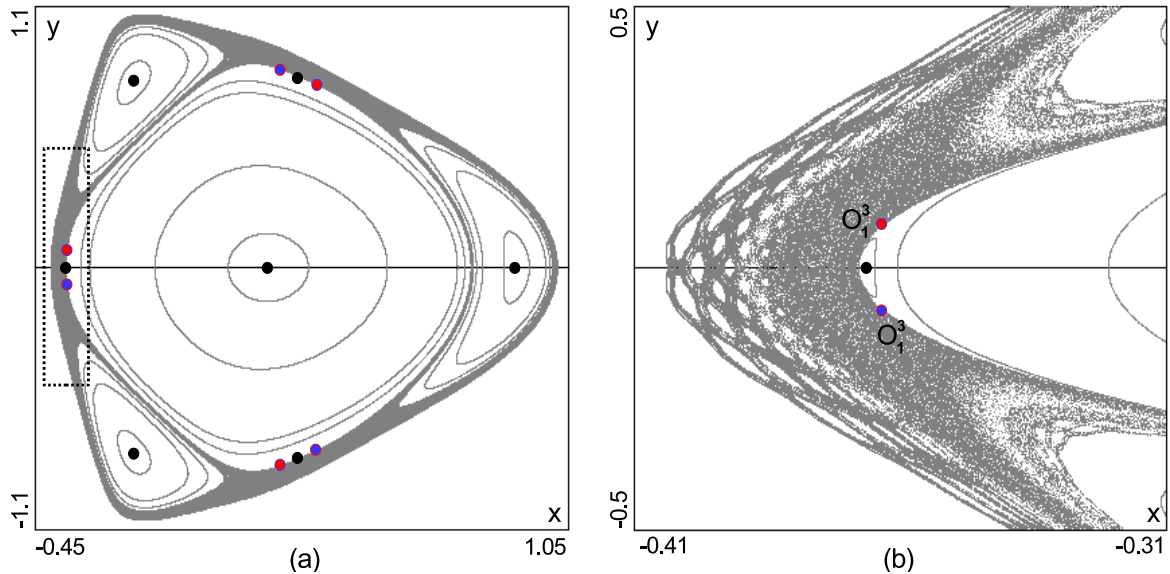}}
\end{minipage}
\caption{{Phase portrait and the zoomed fragment for the map~$\tilde{H}_3^-(\varepsilon)$ in~\eqref{eq:rev2cub} for~$M_1 = -0.364, M_2 = -0.5$ and~$\varepsilon = 0.3$. For the convenience, the phase portrait is rotated by~$\pi/4$. In this representation, the horizontal axis becomes~$Fix(h)$. The chaotic dynamics (in the gray region) seems conservative (the phase portrait is self-symmetric with respect to the horizontal axis). The orbits~$\hat{S}^1_3$ and~$\hat{S}^2_3$ are the pair of non-conservative saddle 3-periodic orbits with the Jacobians~$J(\hat{S}^1_3) = 0.995<1$ and~$J(\hat{S}^2_3) = 1.005>1$.}}
		\label{fig:HenonChaos}
\end{figure}
	
\section{Numerical evidence of mixed dynamics in the perturbed map~$\tilde{H}_3^-(\varepsilon)$ }\label{sec:md}
	
In this section we provide a numerical evidence of the existence of mixed dynamics in the perturbed map~$\tilde{H}_3^-(\varepsilon)$ in~\eqref{eq:rev2cub} for which in Section~\ref{sec:H_3minus} we show the existence of a pair of saddle 3-periodic orbits~$\hat{S}^1_3$ and~$\hat{S}^2_3$  with the Jacobians~$J<1$ and~$J>1$, respectively. Recall that these orbits appear due to a subcritical reversible pitchfork bifurcation of the symmetric saddle 3-periodic orbit~$S_3$. For better visibility, we take a quite large value of perturbation,~$\varepsilon = 0.3$. In Figure~\ref{fig:HenonChaos} we show the phase portraits of~$\tilde{H}_3^-(\varepsilon)$  for~$M_1 = -0.364$ and~$M_2 = -0.5$. The orbits~$\hat{S}^1_3$ with~$J < 1$,~$\hat{S}^2_3$ with~$J > 1$ and symmetric elliptic orbits are marked by blue,  red and black bold points, respectively. For the convenience, we rotate the phase portraits by~$\pi/4$, then the horizontal axis corresponds to~$Fix(h)$.
	
From Figure~\ref{fig:HenonChaos} it is clear to see that the phase portrait is self-symmetric with respect to the horizontal axis which means that the attractor of the system seems coincident with the repeller. Moreover, we are not able to find periodic sinks and sources nor even nonsymmetric orbits (except for  points~$\hat{S}^1_3$ and~$\hat{S}^2_3$) which would confirm mixed dynamics in~$\tilde{H}_3^-(\varepsilon)$.
	
However, we find a non-transversal heteroclinic cycle of Lamb-Stenkin type~\cite{LambStenkin2004} which connects~$\hat{S}^1_3$ and~$\hat{S}^2_3$. As it was shown in~\cite{LambStenkin2004},  bifurcations of such cycles lead to a reversible mixed dynamics. A schematic representation of this cycle is shown in Figure~\ref{fig:nontrans_cycle}a. The numerically obtained cycle is presented in Figure~\ref{fig:nontrans_cycle}b. From this figure, one can see that the stable and unstable manifolds of~$\hat{S}^1_3$ and~$\hat{S}^2_3$ have both transversal (see also the zoomed region near~$\hat{S}^2_3$ in Figure~\ref{fig:nontrans_cycle}c) and non-transversal (see the zoomed fragment in Figure~\ref{fig:nontrans_cycle}b) intersections. Thus, we can state that the chaotic dynamics presented in the gray zone in Figure~\ref{fig:HenonChaos} is mixed.
	
\begin{figure}[]
\begin{minipage}[h]{1\linewidth}
\center{\includegraphics[width=0.9\columnwidth]{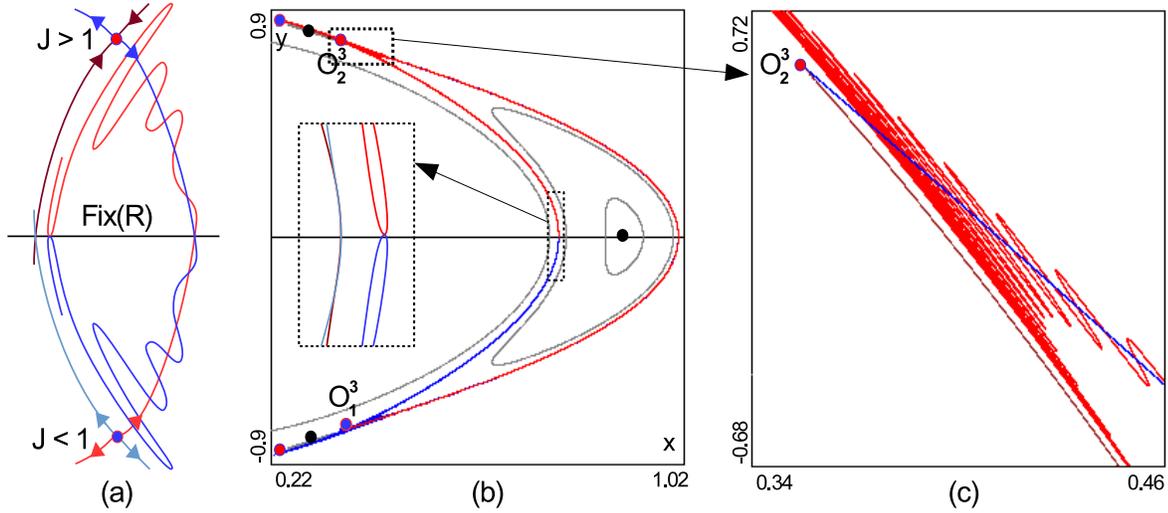}}
\end{minipage}
\caption{{(a) A schematic representation of non-transversal heteroclinic cycle of Lamb-Stenkin type. (b), (c) Non-transversal heteroclinic cycle connected saddles~$\hat{S}^1_3$ and~$\hat{S}^2_3$ in the map~$\tilde{H}_3^-(\varepsilon)$ at~$M_1 = -0.364, M_2 = -0.5$ and~$\varepsilon = 0.3$. A pair of manifolds~$W_1^s$ and~$W_1^u$ intersects transversally, while the other pair~$W_2^s$ and~$W_2^u$ has a quadratic tangency.}}
		\label{fig:nontrans_cycle}
\end{figure}
	
Also in this section we would like to note that the presence of mixed dynamics near elliptic points of two-dimensional reversible maps plays an important role. 
As well-known, the phase portrait near an elliptic point of a two-dimensional reversible diffeomorphism is organized in many details as in the conservative case. There is also a continuum of KAM-curves surrounding the elliptic point. The KAM-curves are separated by resonant zones~\cite{Sevryuk86}. However, the behavior in the resonant zones for reversible maps principally differs from the conservative ones, see, for example, Figures~\ref{fig:elliptic_resonance}a and~\ref{fig:elliptic_resonance}b. In the conservative setting,~$\varepsilon$-orbits can run away from any neighborhood of an elliptic point, i.e. such point is not stable under permanently acting perturbations (Lyapunov instability by~$\varepsilon$-orbits)~\cite{GonTur2017}, see Figure~\ref{fig:elliptic_resonance}a.

On the other hand, in the reversible nonconservative case, as it follows from~\cite{GLRT14, GonTur2017}, it is typical when in resonant zones periodic saddle points alternate with symmetric pairs of sinks and sources, see Figure~\ref{fig:elliptic_resonance}(b). In this case, it is possible when there exist intersecting absorbing domain~$B_A$ and repelling domain~$B_R$ around an elliptic point such that forward as well as backward~$\varepsilon$-orbits of any point, that belongs to the intersection~$B_A \cap B_R$, cannot leave any neighborhood of this elliptic point~\cite{GonTur2017,GonchenkoGK20}. These resonances are called \emph{impassable} or \emph{isolated}. In future papers, we plan to study such resonances for the perturbed reversible non-conservative H\'enon maps.
	
\begin{figure}[t]
\begin{minipage}[h]{1\linewidth}
 \center{\includegraphics[width=0.9\columnwidth]{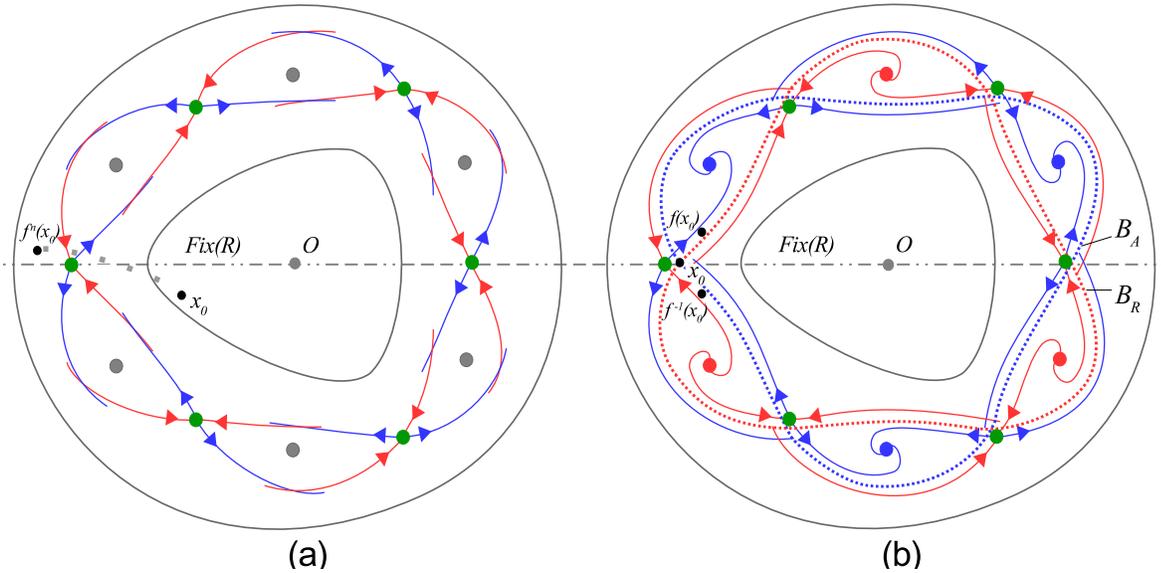}}
\end{minipage}
\caption{{Different types of behavior in resonant zones of a symmetric elliptic point in the conservative (a) and reversible (b) cases. Periodic elliptic orbits are marked by gray bold points, while periodic sinks and sources are colored in blue and red, respectively. In plot (b) the absorbing domain~$B_A$ of a sink orbit (bounded by the blue dashed curves) intersects with the repelling domain of the source orbit (bounded by the red dashed curves). Thus,~$\varepsilon$-orbits of any point belonging to this intersection cannot leave the resonant zone with neither forward nor backward iterations (isolated resonance). It means that an elliptic point of a typical two-dimensional reversible  diffeomorphism is stable under permanently acting perturbations (Lyapunov stability by~$\varepsilon$-orbits).}}
		\label{fig:elliptic_resonance}
\end{figure}

\section*{Conclusions}
	
In the present paper we have obtained a series of new results devoted to the structure of the 1:3 resonance in conservative cubic H\'enon maps and 
their reversible perturbations. 
%The results of the present paper are diverse. 
First, we have used two methods to construct reversible non-conservative perturbations of the conservative cubic H\'enon maps. The first method provides perturbations of the second iteration of the inverse cubic H\'enon maps in the so-called cross-form, while the second more delicate method gives perturbations of the cubic H\'enon maps themselves. In both cases, we have proved that the resulting perturbed perturbations preserve reversibility. Second, we have considered the conservative cubic H\'enon maps~$H^+_3$ and~$H^-_3$ as examples and studied the influence of reversible non-conservative perturbations on the structure of bifurcations of the 1:3 resonance. We have provided a detailed analysis of these bifurcations in the perturbed maps. We have focused on local symmetry-breaking bifurcations which have led to the appearance of nonsymmetric orbits. These bifurcations are reversible pitchfork bifurcations of 3-periodic orbits, and we have found the domains of parameters corresponding to nonsymmetric orbits. Moreover, for perturbations of~$H^+_3$, there appear nonsymmetric asymptotically stable and completely unstable 3-periodic orbits, while in perturbed~$H^-_3$ there emerge two nonsymmetric saddle 3-periodic orbits, one with the Jacobian greater than 1 and the other with the Jacobian less than 1. The presence of these nonsymmetric orbits leads to the existence of mixed dynamics in maps with cubic homoclinic tangencies whose first return maps are~$H^\pm_3$. Third, for the unperturbed conservative maps~$H_3^+$ and~$H_3^-$ with~$M_1=0$, we have demonstrated that all~$p:q$ resonances are degenerate when~$q>3$ is odd. And finally, we have provided a numerical evidence of mixed dynamics in perturbed~$H^-_3$, since a heteroclinic configuration between these nonsymmetric saddle orbits implies  the emergence of a Lamb-Stenkin non-transversal heteroclinic cycle.
	
These results can be used for further study of mechanisms of the appearance of mixed dynamics after a break-down of conservative dynamics. 
As pointed out in the introduction, the maps~(\ref{eq:HenonMapCubic}) are related to the study of cubic homoclinic tangencies and the phenomenon of mixed dynamics, the third (and the last) type of chaos. 
We have shown some global and local mechanisms of its emergence. The global ones are connected with the presence of homoclinic and heteroclinic cycles of different kinds, 
while one of the interesting local mechanisms is related to bifurcations of resonances among which we highlight the 1:3 resonance.
It is easy to associate the bifurcation structure of the cubic H\'enon maps with the bifurcations which take place near the cubic homoclinic tangencies~\cite{GonGonOvs2017}. In the present paper we have proposed a local mechanism corresponding to the occurrence of degenerate resonances in the cubic H\'enon maps. In the reversible context, symmetry-breaking pitchfork bifurcations of 3-periodic orbits lead to the appearance of pairs of nonsymmetric and non-conservative periodic orbits (periodic sinks and sources, periodic saddles with the Jacobians greater and less than 1) near the degenerate 1:3 resonant point. In this regard, we think that it is of great importance to consider the problem of local 1:3 resonance, especially the degeneracy~$a_{02}=0$ in~(\ref{eq:1p3nf}), and the accompanying symmetry-breaking bifurcations in general reversible maps, since degenerate resonances in reversible systems are the main local mechanism of the appearance of mixed dynamics.
It is also worth mentioning that the similar problem for the 1:4 resonance, the structure of bifurcations associated with fixed points with eigenvalues~$e^{\pm \pi/2}=\pm i$ and, consequently, 4-periodic orbits, is of great interest as well. An exhaustive study of 1:4 resonance for~(\ref{eq:HenonMapCubic}) was done in~\cite{GonGonOvs2017, GGOV18}, see also~\cite{MGon05}. It was also established in these works that for some~$(M_1,M_2)$, the 1:4 resonance can be degenerate and the 4-periodic orbits are subject to pitchfork bifurcations. The study of reversible non-conservative perturbations for this case is planned in a forthcoming paper.
	
\section*{Acknowledgments}
	
The authors thank S.~V.~Gonchenko, D.~Turaev, and K.\,Safonov for fruitful discussions. 
This paper was supported by the RSF grant No. 17-11-01041. Numerical experiments in Section~\ref{sec:md} were supported by the Laboratory of Dynamical Systems and Applications NRU HSE, of the Russian Ministry of Science and Higher Education (Grant No. 075-15-2019-1931). Section~\ref{sec:1p3resUnpert} was supported by the RSF grant No. 19-11-00280.
M.\,Gonchenko is partially supported by Juan de la Cierva-Incorporaci\'on fellowship IJCI-2016-29071 and the Spanish grant PGC2018-098676-B-I00 (AEI/FEDER/UE). A. Kazakov and E. Samylina also acknowledge RFBR project 18-29-10081 and the Theoretical Physics and Mathematics Advancement Foundation “BASIS” for financial support of scientific investigations.

\end{document}